\documentclass[11pt, reqno,amsmath,amsthm,amssymb,amscd]{amsart}
\usepackage{mathrsfs,amssymb, amscd,amsmath,amsthm}
\usepackage[enableskew,vcentermath]{youngtab}
\usepackage{multicol}\multicolsep=0pt
\usepackage{pstricks,pst-node}
\usepackage[enableskew,vcentermath]{youngtab}

\usepackage{amsthm}
\usepackage{amssymb}
\usepackage{amsmath}

\providecommand{\bysame}{\leavevmode ---\ } \providecommand{\og}{``}
\providecommand{\fg}{''} \providecommand{\smfandname}{and}
\providecommand{\smfedsname}{\'eds.}
\providecommand{\smfedname}{\'ed.}
\providecommand{\smfmastersthesisname}{M\'emoire}
\providecommand{\smfphdthesisname}{Th\`ese}

\def\1{\hbox{1\kern-.35em\hbox{1}}}

\newtheorem{theorem}{Theorem}[section]
\newtheorem*{theorem*}{Theorem}
\newtheorem{lemma}[theorem]{Lemma}
\newtheorem{proposition}[theorem]{Proposition}
\newtheorem*{proposition*}{Proposition}
\newtheorem{corollary}[theorem]{Corollary}
\newtheorem{coro}[theorem]{Corollary}

\theoremstyle{definition}
\newtheorem{definition}[theorem]{Definition}

\theoremstyle{remark}
\newtheorem{remark}[theorem]{Remark}

\numberwithin{equation}{section}

\newcommand{\bea}{\begin{eqnarray}}
\newcommand{\eea}{\end{eqnarray}}
\newcommand{\be}{\begin{eqnarray*}}
\newcommand{\ee}{\end{eqnarray*}}

\def\t{{\mathfrak t}}
\def\s{{\mathfrak s}}

\newcommand{\Hom}{{\rm Hom}}

\def\SUBS#1{}

\def\a{\alpha}

\def\d{\delta}

\def\m{\mathfrak m}
\def\n{\mathfrak n}
\def\End{\text{End}}

\def\Std{\mathcal T^{s}}

\def\Hom{{\rm Hom}}

\def\s{\mathfrak s}

\def\Set{{\rm Set}}

\def\a{\alpha}

\def\wtt{{\rm wt}}

\def\Set[#1]#2|#3|{\Big\{\ #2\ \Big| \
           \vcenter{\hsize #1mm\centering #3}\Big\}}

\numberwithin{equation}{section}
\title[Decomposition matrices  of Birman-Murakami-Wenzl algebras  ]{Decomposition matrices  of Birman-Murakami-Wenzl algebras}

\author{Hebing Rui  {\normalfont \smfandname} Linliang \vspace*{-8pt}Song}
\address{H.~Rui: Department of Mathematics,  East China Normal
University, Shanghai, 200241, China} \email{hbrui@math.ecnu.edu.\vspace*{-5pt}cn}
\address{L.~Song: Mathematics and Science College,  Shanghai Normal University, Shanghai, 200234, China }
\email{song51090601020@163.com}
\thanks{Rui is supported by NSFC (grant no.~11025104)}
\begin{document}
\baselineskip15pt
\begin{abstract} In this paper, we calculate decomposition matrices of  the Birman-Murakami-Wenzl algebras  over $\mathbb C$.
 \end{abstract}
\sloppy \maketitle

\section{Introduction} One of key problems  in studying structure of  a finite dimensional algebra is to determine its semisimple quotient. This leads to calculate dimensions of its  simple modules.
 In this paper, we address this problem on a Birman-Murakami-Wenzl algebra over $\mathbb C$ by determining its decomposition numbers.

Recall that  Birman-Murakami-Wenzl algebras  are  unital  associative $R$-algebras introduced in~\cite{BM, Mu}, where $R$ is a  commutative ring  containing $1$ and invertible elements $\varrho, q$ and $q-q^{-1}$.
Suppose $R$ is a field $\kappa$. If  $\varrho\not\in \{ q^a, -q^a\mid a\in \mathbb Z\}$,
Rui and Si~\cite{RSsin} proved that
    $\mathscr B_r(\varrho, q)$ is Morita equivalent to $\bigoplus_{i=0}^{\lfloor r/2\rfloor} \mathscr H_{r-2i}$
    where $\mathscr H_{r-2i}$ are  Hecke algebras associated to  symmetric groups $\mathfrak S_{r-2i}$.
    In non-semisimple cases and $\kappa=\mathbb C$, by Ariki's result on decomposition numbers of Hecke algebras  in \cite{Ari},  decomposition numbers of   $\mathscr B_r(\varrho, q)$
are determined by the values of certain inverse Kazhdan-Lusztig polynomials at $q=1$
associated to some extended affine Weyl groups of type $A$. If $\varrho\in \{q^a, -q^a\}$ for some  $a\in \mathbb Z$ and if  $q^2$ is not a root of unity, Rui and Si
classified blocks of $\mathscr B_r(\varrho, q)$ over $\kappa$ \cite{RSbloc}. Via such results together with Martin's arguments on the decomposition matrices of Brauer algebras over $\mathbb C$ in \cite{Mart},
 Xu showed  that $\mathscr B_r(\varrho, q)$ is multiplicity-free over $\mathbb C$~\cite{Xu}. In other words, the multiplicity of a simple module in a cell (or standard) $\mathscr B_r(\varrho, q)$-module  is either $1$ or $0$ if $\kappa$ is $\mathbb C$.

 The aim of this paper is to calculate  decomposition matrices  of  $\mathscr B_r(\varrho, q)$ over $\mathbb C$ when
 $\varrho\in \{ -q^a, q^a\} $ for some  $a\in \mathbb Z$ and $q^2$ is a root of unity.
 In this case, it is enough to assume  either $\varrho=-q^{2n+1}$ or $\varrho =q^n$ for some $n\in \mathbb Z$ such that  $n\gg 0$.
In the first case, Hu~\cite{Hu} proved that there is an  integral Schur-Weyl duality between $\mathscr B_r(-q^{2n+1}, q)$ and  the quantum group $\mathbf U(\mathfrak{sp}_{2n})$ associated to $\mathfrak{sp}_{2n}$.
In particular, Hu proved that  $\mathscr B_r(-q^{2n+1}, q)$ is isomorphic to $\End_{\mathbf U(\mathfrak {sp}_{2n} )}(V^{\otimes r})$ if $n\ge r$, where $V$ is the natural representation of $\mathbf U(\mathfrak{sp}_{2n})$.
Moreover, Hu's arguments in \cite{Hu}  can be used smoothly to prove that   $\mathscr B_r(q^n, q)$ is isomorphic to $\End_{\mathbf U(\mathfrak g)}(V^{\otimes r})$ if $\lfloor \frac {n+1}{2}\rfloor> r$, where $V$ is the natural representation of $\mathbf U(\mathfrak{so}_{n+1})$. Motivated by our work on quantized walled Brauer algebras in \cite{RSong1}, we establish  an  explicit relationship between  decomposition numbers of  $\mathscr B_r(\varrho,  q)$ with $\varrho\in \{-q^{2n+1}, q^{n}\}$ and the multiplicities of Weyl modules in indecomposable direct summands of $V^{\otimes r}$ (called partial tilting modules). When the ground field is $\mathbb C$ and $e$, the order of  $q^2$ is bigger than $29$,  such multiplicities have been given  in~\cite{Soe}\footnote{Soergel needs the equivalence of categories between modules for quantum groups at roots of unity and corresponding module categories for Kac-Moody algebras in \cite{Soe1}. Due to \cite{KL1}, this equivalence  is only proved when $e \geq 29$. Thanks Professor H.H.~Andersen for his explanation.}.  Suppose $e=\infty$. By arguments similar to those in \cite{DT},  the decomposition matrices of $\mathscr B_r(\varrho, q)$ are the same as those for $\mathscr B_r(\varrho, q)$ with $e\gg 0$. In particular, we recover   \cite[Theorem~5.6]{Wen2} by assuming that $\varrho=-q^{2n+1} $.

We organize this paper   as follows. In section~2, after recalling some well known results on quantum groups,  we use Hu's arguments in \cite{Hu} to  show that  $\mathscr B_r(q^n, q)$ is isomorphic to  $\End_{\mathbf U(\mathfrak {so}_{n+1})} (V^{\otimes r})$ if $\lfloor\frac{ n+1} {2}\rfloor >r$,  where $V$ is the natural representation of $\mathbf U(\mathfrak {so}_{n+1})$.  In section~3,  we prove that $V^{\otimes r}$ is self-dual as $(\mathbf U(\mathfrak g), \mathscr B_r(\varrho, q))$-bimodule where $\mathfrak g\in \{\mathfrak{sp}_{2n}, \mathfrak {so}_{2n},  \mathfrak {so}_{2n+1}\}$ and  $\varrho$ is given  in \eqref{varrho}.  In section~4, we classify highest weight vectors of $V^{\otimes r}$.  This leads us to  establish  an explicit
 relationship between decomposition numbers of  $\mathscr B_r(\varrho, q)$ with some special parameters $\varrho$ in \eqref{varrho} and the multiplicities of Weyl modules in indecomposable tilting modules for $\mathbf U(\mathfrak g)$. So, we can use Soergel's results in \cite{Soe, Soe1} to calculate decomposition numbers of  Birman-Murakami-Wenzl algebras over $\mathbb C$. Together with some previous results, we settle the problem on decomposition matrices of $\mathscr B_r(\varrho, q)$ over $\mathbb C$ under the assumption  $e\ge 29$.

\section{Schur-Weyl duality in classical types}
 Throughout, let  $\mathcal A=\mathbb Z[v, v^{-1}]$ with the quotient field $\mathbb Q(v)$ where $v$ is an indeterminate. For any $n\in \mathbb N$,
let
\begin{equation}\label{qinte}  [n]=\frac{v^n-v^{-n}}{v-v^{-1}}. \end{equation}
 For $m, n, d\in \mathbb N$,  following \cite{Lusz1}, define
 \begin{equation} \label{gus}  [n]^{!}_d:=\prod_{i=1}^n  \frac{v^{di}-v^{-di}}{v^d-v^{-d}}, \quad
 \left[{m+n\atop n}\right]_d=\frac{[m+n]^!_d}{[m]^!_d [n]^!_d}\in \mathcal A.\end{equation}
The Cartan matrix is an $n\times n$ matrix $A=(a_{ij})$  with entries $a_{ij}\in \mathbb Z$, $1\le i,j\le n$ such that $(d_i a_{ij})$ is symmetric and positive definite, where $d_{i}\in \{1, 2, 3\}$ and
  $a_{ii}=2$ and $a_{ij}\le 0$ for $i\neq j$.
The quantum group  $U_v$ associated with $A$ is  the associative $\mathbb Q(v)$-algebra generated by $\{e_i, f_i, k_i^{\pm 1}\mid  1\le i\le n\}$ subject to the relations:
\begin{equation}\label{qgroup}\begin{cases} & k_i k_i^{-1}=1=k_{i}^{-1} k_i, \quad k_i k_j=k_j k_i,\\
&  k_i e_j k_i^{-1}=v^{d_ia_{ij}} e_j,\\
& k_i f_j k_i^{-1}=v^{-d_ia_{ij}} f_j,\\
&  e_i f_j-f_j e_i=\delta_{ij} \frac{k_i-k_i^{-1}}{v^{d_i} -v^{-d_i}}, \\
&\sum_{s=0}^{1-a_{ij} } (-1)^s [{1-a_{ij}\atop s}]_{d_i} e_i^{1-a_{ij}-s} e_j e_i^s=0, \text{if   $i\neq j$,}\\
&\sum_{s=0}^{1-a_{ij} } (-1)^s [{1-a_{ij} \atop s}]_{d_i} f_i^{1-a_{ij}-s} f_j f_i^s=0, \text{if   $i\neq j$,}\\
\end{cases}\end{equation}
where $\delta_{ij}$ is the Kronecker delta.
 It is known that $U_v$ is a Hopf algebra with the comultiplication $\Delta$, counit $\epsilon$ and antipode $S$ defined by
\begin{equation}\label{coalg} \begin{aligned}  & \Delta(e_i)=e_i\otimes 1+ k_i\otimes e_i,\  \ \ \epsilon(e_i)=0,\  S(e_i)=-k_i^{-1} e_i,\\
& \Delta(f_i)=f_i\otimes k_i^{-1} + 1\otimes f_i,\  \epsilon(f_i)=0, \ S(f_i)=-f_i k_i,\\
& \Delta(k_i) =k_i\otimes k_i, \ \quad \quad \quad \quad \  \epsilon(k_i)=1,\   S(k_i)= k_i^{-1}. \\ \end{aligned}\end{equation}
For all positive integers  $k$, following \cite{Lusz1}, let \begin{equation}\label{int} e_i^{(k)}=e_i^{k}/[k]^!_{d_i}, \text{ and }  f_i^{(k)}=f_i^{k}/[k]^!_{d_i}.\end{equation}   Then  $U_v$ contains the $\mathcal A$-subalgebra
$\mathbf U$ generated by $\{e_i^{(k)}, f_i^{(k)}, k_i^{\pm 1}\mid   1\le i\le  n, k\in  \mathbb Z^{>0}\}$.
 Further, $\mathbf U$ is a Hopf algebra such that comultiplication, counit  and antipode are obtained from those for $U_v$ by restrictions.

In this paper, we consider quantum groups associated with  complex semisimple Lie algebras $\mathfrak g\in \{\mathfrak{sl}_{n+1}, \mathfrak {so}_{2n+1}, \mathfrak {sp}_{2n}, \mathfrak {so}_{2n}\}$.
 According to \cite{Bou}, we have
 the root systems for $\mathfrak g$ so that $\epsilon_1, \epsilon_2, \cdots, \epsilon_n$  are orthonormal and if $\mathfrak g=\mathfrak{sl}_{n+1}$,  also include $\epsilon_{n+1}$.
Let $\Pi=\{\a_i\mid 1\le i\le n\}$, where $\alpha_i=\epsilon_i-\epsilon_{i+1}$ for $1\le i\le n-1$ and
\begin{equation} \label{simplerootn} \a_n=\begin{cases} \epsilon_n-\epsilon_{n+1}, &\text{if
$\mathfrak g=\mathfrak{sl}_{n+1}$,}\\
\epsilon_n,    &\text{if
$\mathfrak g=\mathfrak{so}_{2n+1}$,}\\
2\epsilon_n,    &\text{if
$\mathfrak g=\mathfrak{sp}_{2n}$,}\\
\epsilon_{n-1}+\epsilon_{n},    &\text{if
$\mathfrak g=\mathfrak{so}_{2n}$.}\\
\end{cases}\end{equation}
Then $\Pi$ is a  set of  simple roots associated with $\mathfrak g$.
 The weight lattice $P$  is $\oplus_{i=1}^n\mathbb Z \omega_i$, where $\omega_i$'s are  fundamental weights  given by
 \begin{enumerate} \item $\omega_i=\epsilon_1+\cdots+\epsilon_i-\frac {i}{n+1}(\epsilon_1+\cdots+\epsilon_{n+1})$, $1\le i\le n$ if $\mathfrak g=\mathfrak{sl}_{n+1}$,
 \item $\omega_i=\epsilon_1+\cdots+\epsilon_i$, $1\le i\le n-1$, and $\omega_n=\frac{1}{2} (\epsilon_1+\cdots+\epsilon_n)$  if $\mathfrak g=\mathfrak{so}_{2n+1}$,
 \item $\omega_i=\epsilon_1+\cdots+\epsilon_i$, $1\le i\le n$, if $\mathfrak g=\mathfrak{sp}_{2n} $,
 \item $\omega_i= \epsilon_1+\cdots+\epsilon_i$, $1\le i\le n-2$, $\omega_{n-1}=\frac{1}{2} (\epsilon_1+\cdots+\epsilon_{n-1}-\epsilon_n)$ and
  $\omega_n=\frac{1}{2} (\epsilon_1+\cdots+\epsilon_n)$  if $\mathfrak g=\mathfrak{so}_{2n}$. \end{enumerate}
Let $P^+=\oplus_{i=1}^n \mathbb N \omega_i$. Then $P^+$ is the set of all dominant integral weights.
For $\alpha_i, \a_j\in \Pi$, let  $$a_{ij}=\langle \alpha_i, \a_j\rangle=2(\alpha_i, \a_j)/(\a_i, \a_i),$$ where  $( \ , \ )$ is the symmetric bilinear form such that $(\epsilon_i, \epsilon_j)=\d_{ij}$.
The Cartan matrix $A$  associated with $\mathfrak g$   is  the $n\times n$ matrix $(a_{ij})$, which is the transpose of that in \cite{Bou}. So, the quantum groups $U_v(\mathfrak g)$ associated with $\mathfrak g$ defined in~\eqref {qgroup}  are the same as those  in \cite{ha}. They are associative algebras over $\mathbb Q(v)$   such that  $v=q^{1/2}$ if $\mathfrak g=\mathfrak{so}_{2n+1}$ and
$v=q$, otherwise.  Further,
\begin{enumerate} \item $d_i=1$, $1\le i\le n$ if $\mathfrak g\in \{\mathfrak{sl}_{n+1}, \mathfrak{so}_{2n}\}$,
\item $d_i=2$, $1\le i\le n-1$ and $d_n=1$ if $\mathfrak g=\mathfrak{so}_{2n+1}$,
\item $d_i=1$, $1\le i\le n-1$ and $d_n=2$ if $\mathfrak g=\mathfrak{sp}_{2n}$.\end{enumerate}
 If $M$ is a $ U_v(\mathfrak g)$-module, let \begin{equation}\label{weight}  M_{\lambda}=\{m\in M\mid k_i m=v_i^{\langle\lambda,  \alpha_i\rangle}m, 1\le i\le n\},  \ \  \text{for any $\lambda\in P$,}\end{equation}
where $v_i=v^{d_i}$ and ${\langle\lambda,  \alpha_i\rangle}=2(\lambda, \alpha_i)/(\alpha_i, \alpha_i)$.
Then $M_\lambda$ is called  the weight space of $M$ with respect to the weight   $\lambda$ if $M_\lambda\neq 0$.
  For any field $\kappa$ which is an $\mathcal A$-algebra, let $\mathbf U_{\kappa}(\mathfrak g)=\mathbf U(\mathfrak g)\otimes_{\mathcal A} \kappa$, where $\mathbf U(\mathfrak g)$ is the $\mathcal A$-form of $U_v(\mathfrak g)$.
If $M$ is a $\mathbf U_\kappa(\mathfrak g)$-module, the weight space of $M$ can be defined  by base change. Later on, we write $$\wtt(m)=\lambda \ \text{if $m\in M_\lambda$.}$$
In the remaining part of this paper,  we always assume that
\begin{equation}\label{N}  N=\begin{cases} n, & \text{ if $\mathfrak g=\mathfrak {sl}_{n} $,}\\  2n+1, & \text{ if $\mathfrak g=\mathfrak {so}_{2n+1} $,}\\
 2n, & \text{ if $\mathfrak g=\mathfrak{sp}_{2n}, \mathfrak {so}_{2n}$.}\\
  \end{cases}\end{equation}
If $\mathfrak g=\mathfrak{so}_{2n+1}$,  we write $ i'=2n+2-i$,   $1\le i\le n+1$,
and hence \begin{equation} \label{tot1} 1<2<\cdots< n<n+1< n'<\cdots < 1'.\end{equation}
If $\mathfrak g\in \{\mathfrak{sp}_{2n}, \mathfrak{sp}_{2n}\}$, we write $  i'=2n-i+1$,   $1\le i\le n$, and hence
\begin{equation} \label{tot2} 1<2<\cdots< n< n'<\cdots < 1'.\end{equation}
In any case, we set $ k''=k$ for any $1\le k\le n$. If   $\mathfrak g=\mathfrak{so}_{2n+1}$,
$(n+1)'=n+1$.
Unless otherwise state, we always assume that $\kappa$ is a field which is an $\mathcal A$-algebra  such that $v$ acts on $\kappa$ via  $q\in \kappa^*$  (resp., $q^{1/2}\in \kappa^*$ if $\mathfrak g=\mathfrak{so}_{2n+1}$).
\begin{lemma}\label{basicl} Let $V=\bigoplus_{i=1}^N \kappa v_i$, where  $N$ is given  in \eqref{N}. Then $V$ is a left $\mathbf U_\kappa(\mathfrak g)$-module such that the following conditions hold.
 \begin{enumerate}  \item  If $\mathfrak g=\mathfrak{sl}_{n}$, then  \begin{itemize}\item [(a)]  $ e_i v_k  =\delta_{k, i+1} v_i$,
 \item [(b)] $ f_i v_k =\delta_{i,k} v_{i+1}$,
 \item [(c)] $k_iv_k= q^{\epsilon}  v_k$, where
   $\epsilon=1$ (resp., $-1$) if $k=i$ (resp.,~$i+1$) and $\epsilon=0$ in the remaining cases.\end{itemize}
   \item  If $\mathfrak g=\mathfrak{so}_{2n+1}$, then  for $i\neq n$,  \begin{itemize}
   \item [(a)]  $e_i v_{i+1}= v_i$,  $e_i v_{i'}= -v_{(i+1)'}$ and  $e_i v_k=0$, otherwise,
   \item [(b)] $f_i v_i= v_{i+1}$, $f_iv_{(i+1)'}=-v_{{i'}}$, and  $f_i v_k=0$ otherwise,
    \item [(c)]  $k_i v_k= q v_k $ (resp.,   $  q^{-1}v_{k}$ ) if $k\in \{i,  (i+1)'\}$ (resp., $k\in \{i+1, i'\}$), and $k_i v_k=v_k$, otherwise,
    \item [(d)]  $e_n v_{n+1} = v_n$, $e_{n}v_{n'} =- q^{-1/2}  v_{n+1}$,  and   $e_n v_k=0$, otherwise,
       \item [(e)]  $f_n v_n = [2]_{ q^{1/2}}  v_{n+1}$, $f_n v_{n+1}  =- q^{1/2} [2]_{ q^{1/2}}  v_{n'}$ and   $f_n v_k=0$, otherwise,
   \item  [(f)] $k_n v_n= q v_{n}$, $k_n v_{n'}=  q^{-1}v_{ n'}$ and   $k_n v_k=v_k$, otherwise.
 \end{itemize}
 \item  If $\mathfrak g=\mathfrak{sp}_{2n}$, then  for $i\neq n$,   \begin{itemize}\item [(a)] $e_i v_k$, $f_iv_k$ and $k_i v_k$   satisfy the formulae in
    (2a)--(2c),  respectively,
   \item [(b)] $ e_n v_{n'}=v_n$ and $e_n v_k=0$, otherwise,
    \item [(c)] $f_n v_n=v_{ n'}$ and $f_n v_k=0$, otherwise,
    \item [(d)] $k_n v_n=  q^2 v_{n}$,  $k_n v_{n'} =  q^{-2}  v_{n'}$ and  $k_n v_k=v_k$, otherwise.
    \end{itemize}
\item    If $\mathfrak g=\mathfrak{so}_{2n}$, then  for $i\neq n$,
 \begin{itemize}\item [(a)]  $e_i v_k$, $f_iv_k$ and $k_i v_k$ satisfy the formulae in
    (2a)--(2c), respectively,
\item [(b)]  $e_n v_{ n'} =v_{n-1}$, $e_n v_{(n-1)'}=-v_{n}$ and $e_n v_k=0$, otherwise,
 \item [(c)]  $f_n v_{n-1} =v_{ n'}$, $f_n v_{n} =-v_{(n-1)'} $ and $f_n v_k=0$, otherwise,
\item[(d)]  $k_n v_k=  q v_{k}$ (resp.,  $ q^{-1}v_{k}$)  if $k\in \{n-1, n\}$ (resp., $ \{(n-1)',  n'\}$), and
$k_n v_k=v_k$, otherwise.
\end{itemize}
\end{enumerate}
 \end{lemma}
\begin{proof} When $\kappa=\mathbb Q(v)$,  Lemma~\ref{basicl}(1)--(4) have been given in \cite[(4.16)]{ha}\footnote{If $\mathfrak g=\mathfrak{so}_{2n+1}$ and  $\kappa=\mathbb Q(q^{1/2})$,
 there is a difference between (2) and that in \cite[(4.16)]{ha}, where  Hayashi defined $e_n v_k=0$ unless $k\in \{n+1, n'\}$ and $e_n v_{n+1}=q^{1/2} v_n$, $e_n v_{n'}=-v_{n+1}$ and   $f_n v_k=0$ unless $k\in \{n, n+1\}$ and $f_n v_{n+1}=- v_{n'}$, $f_n v_{n}=q^{-1/2}v_{n+1}$. In this case, $(e_nf_n-f_ne_n) (v_n)\neq (\frac{k_n-k_n^{-1}}{q^{1/2}-q^{-1/2}})(v_n)$.}.
 In general, since  $V$ has a $\mathcal A$-lattice spanned by $\{v_i\mid 1\le i\le N\}$, which is a left $\mathbf U(\mathfrak g)$-module, the result follows from arguments on base change.  \end{proof}

The $\kappa$-space $V$ in Lemma~\ref{basicl} is known as the \textsf{natural representation} of $\mathbf U_\kappa(\mathfrak g)$.

\begin{coro} Let  $V=\bigoplus_{i=1}^N \kappa v_i$ be  the {natural representation} of $\mathbf U_\kappa(\mathfrak g)$.\begin{enumerate}
 \item If $\mathfrak g=\mathfrak{ sl}_n$, then $\wtt (v_i)=\epsilon_i-\frac 1 n\sum_{i=1}^n \epsilon_i$. \item If
 $\mathfrak g=\{\mathfrak{so}_{2n+1}, \mathfrak{sp}_{2n},  \mathfrak{so}_{2n}\}$  then $\wtt (v_i)=-\wtt (v_{ i'})=\epsilon_i$, $1\le i\le n$. Further, if  $\mathfrak g=\mathfrak{so}_{2n+1}$, then
  $\wtt (v_{n+1})=0$.\end{enumerate}
 \end{coro}
 \begin{proof} The result follows from  Lemma~\ref{basicl}, immediately.
 \end{proof}
 In the remaining part of this paper, we always assume that
 \begin{equation} \label{rho} \rho=\begin{cases} (n-\frac1 2, n-\frac 3 2, \cdots, \frac 1 2, 0, -\frac 1 2, \cdots, \frac 3 2-n,  \frac 1 2-n), &\text{ $\mathfrak g=\mathfrak{so}_{2n+1} $,}\\
(n-1, n-2, \cdots, 1,  0, 0, -1\cdots, 2-n,  1-n), &\text{ $\mathfrak g=\mathfrak{so}_{2n} $,}\\
 (n, n-1, \cdots, 1,  -1, \cdots, 1-n,  -n), &\text{ $\mathfrak g=\mathfrak{sp}_{2n} $.}\\
 \end{cases}\end{equation}

\begin{coro}\label{1dim} Let  $V=\bigoplus_{i=1}^N \kappa v_i$ be  the {natural representation} of $\mathbf U_\kappa(\mathfrak g)$.
The $1$-dimensional $\kappa$-subspace of $V^{\otimes 2}$ generated by  $\alpha= \sum_{k=1}^{N}q^{\rho_{{k'}}}\varepsilon_{{k'}}v_{k}\otimes v_{{k'}} $  is a left $\mathbf U_\kappa(\mathfrak g)$-module
 where $\rho$  is   defined in \eqref{rho}  and $\varepsilon_i=1$ unless $\mathfrak g=\mathfrak {sp}_{2n}$ and  $n+1\le i\le 2n$. In the later case,  $ \varepsilon_i= -1$.
\end{coro}
\begin{proof} Obviously, $k_i \alpha=\alpha$,  $1\le i\le n$. Suppose $1\le i\le n-1$.  By Lemma 2.1,
$$\begin{aligned}e_i\alpha=&q^{\rho_{(i+1)'}}\varepsilon_{(i+1)'}e_iv_{i+1}\otimes v_{(i+1)'}+q^{\rho_i}\varepsilon_ie_iv_{{i'}}\otimes v_i\\ &
 +q^{\rho_{i'}}\varepsilon_{i'}k_iv_i\otimes e_iv_{i'}+q^{\rho_{i+1}}\varepsilon_{i+1}k_iv_{(i+1)'}\otimes e_iv_{i+1}\\
=& (q^{\rho_{(i+1)'}}\varepsilon_{(i+1)'}-q^{\rho_{i'}+1}\varepsilon_{i'})v_{i}\otimes v_{(i+1)'}+(q^{\rho_{i+1}+1}\varepsilon_{i+1}-q^{\rho_i}\varepsilon_i)v_{(i+1)'}\otimes v_i,
\end{aligned}$$
Since $q^{\rho_{(i+1)'}}\varepsilon_{(i+1)'}-q^{\rho_{i'}+1}\varepsilon_{i'}= q^{\rho_{i+1}+1}\varepsilon_{i+1}-q^{\rho_i}\varepsilon_i=0$, $ 1\le i\le n-1$,   $e_i \alpha=0$.

If  $\mathfrak g=\mathfrak{sp}_{2n}$, by Lemma~2.1,
$$e_n\alpha=q^{\rho_n}\varepsilon_n v_{n'}\otimes v_n+q^{\rho_{n'}}\varepsilon_{n'}k_n v_n\otimes e_n v_{n'} =qv_n\otimes v_n-qv_n\otimes v_n=0.$$

If  $\mathfrak g=\mathfrak{so}_{2n}$,  by Lemma~2.1,
$$\begin{aligned}e_n\alpha=&q^{\rho_n} e_nv_{n'}\otimes v_n+q^{\rho_{n-1}}e_n v_{(n-1)'}\otimes  v_{n-1}+q^{\rho_{n'}}k_nv_n\otimes e_nv_{n'}
\\ &  +q^{\rho_{(n-1)'}}k_nv_{n-1}\otimes e_nv_{(n-1)'}\\
 =&(q^{\rho_n}-q^{\rho_{(n-1)'}+1})v_{n-1}\otimes v_n+(q^{\rho_{n'}+1}-q^{\rho_{n-1}})v_n\otimes v_{n-1}=0.\end{aligned}$$

 If $\mathfrak g=\mathfrak{so}_{2n+1}$,  by Lemma~2.1,
$$\begin{aligned}e_n\alpha=&q^{\rho_{n+1}} e_nv_{n+1}\otimes v_{n+1}+q^{\rho_{n}}e_n v_{n'}\otimes  v_{n} +q^{\rho_{n+1}}k_nv_{n+1}\otimes e_nv_{n+1}
\\ &  +q^{\rho_{n'}}k_nv_{n}\otimes e_nv_{n'}\\
 =&(q^{\rho_{n+1}}-q^{\rho_{n'}+\frac 1 2})v_{n}\otimes v_{n+1}+(q^{\rho_{n+1}}-q^{\rho_{n}-\frac 1 2})v_n\otimes v_{n-1}=0.\end{aligned}$$
In any case, we have   $e_i \alpha=0$, $1\le i\le n$. Finally, one can check $f_i \alpha=0$, $1\le i\le n$.
\end{proof}

Let $E_{ij}$'s  be the matrix units.
 Consider  the operator
  \begin{equation}\label{rmat}\begin{aligned}  \breve{R}= & \sum_{i\neq  i'} (qE_{ii}\otimes  E_{ii}+  q^{-1} E_{i i'}\otimes E_{ i'i})  +\sum_{i\neq j, j'} E_{ij}\otimes E_{ji}\\
 &+(q-q^{-1})\sum_{i>j} (E_{jj}\otimes E_{ii}-q^{\rho_i-\rho_j}\varepsilon_i\varepsilon_j E_{j i'}\otimes E_{ j' i})+X,\\ \end{aligned} \end{equation}
where $\varepsilon_i$'s  (resp., $\rho$) are defined in Corollary~\ref{1dim} (resp., \eqref{rho}),  and $X$ is $E_{n+1,n+1}^{\otimes 2}$ if $\mathfrak g=\mathfrak {so}_{2n+1}$ and $0$,  otherwise.
 As in \eqref{tot1}--\eqref{tot2},  we go on   identifying $\{1, 2, \cdots,  2',  1'\}$ with $\{1, 2, \cdots, N\}$. Let
    $\delta=q-q^{-1}$.

\begin{lemma}\label{basic2} Let  $V=\bigoplus_{i=1}^N \kappa v_i$ be the {natural representation} of $\mathbf U_\kappa(\mathfrak g)$.
\begin{enumerate}\item  If either  $\mathfrak g\neq \mathfrak {so}_{2n+1}$ or $\mathfrak g=\mathfrak{so}_{2n+1}$ and $(k, \ell)\neq (n+1, n+1)$, then
{\small\begin{equation} \label{form}  (v_k\otimes v_{\ell} )\breve{R} =\begin{cases} q v_{k}\otimes v_{k}, &\text{if $k=\ell$,}\\
v_\ell\otimes v_k, &\text{if $k>\ell, k\neq \ell'$,}\\
q^{-1} v_\ell\otimes v_k-\delta \sum\limits_{i>k} q^{\rho_i-\rho_k} \epsilon_i\epsilon_k v_{ i'}\otimes v_i, &\text{if $k>\ell$, $k=  \ell'$,}\\
v_{\ell}\otimes v_k+\delta v_k\otimes v_{\ell}, &\text{if $k<\ell$, $k\neq \ell'$,}\\
q^{-1} v_{\ell}\otimes v_k+\delta (v_k\otimes v_{\ell}-\sum\limits_{i>k} q^{\rho_i-\rho_k}\epsilon_i\epsilon_k v_{ i'}\otimes v_i), &\text{if $k<\ell$, $k=  \ell'$.}\\
\end{cases}\end{equation}
}
\item  If $\mathfrak g=\mathfrak{so}_{2n+1}$, then $  (v_{n+1} \otimes v_{n+1} )\breve{R}  = v_{n+1}\otimes v_{n+1} -\delta \sum\limits_{i>n+1} q^{\rho_i} v_{ i'}\otimes v_i$, \item
$ \breve{R}-\breve{R}^{-1}=\delta (1-E)$,  where $E: V^{\otimes 2} \rightarrow V^{\otimes 2}$ such that
\begin{equation}\label{ei} (v_k\otimes v_{\ell})E =\begin{cases}\sum_{i=1}^{N} q^{\rho_{i'}-\rho_k}\varepsilon_{ i'}\varepsilon_k v_i\otimes v_{ i'}, &\text{if $k= \ell'$, }\\ 0, &\text{otherwise.}\\

\end{cases}
\end{equation}\end{enumerate}
\end{lemma}
\begin{proof} Easy exercise.\end{proof}

Following \cite{Hu},  we say that $ v_{j_1}\otimes v_{j_2}$  is \textsf{involved} in $(v_{i_1}\otimes v_{i_2})\check{R}$ if it appears in $(v_{i_1}\otimes v_{i_2})\check{R}$ with non-zero coefficient.
For any positive integers $r$ and $N$,  let \begin{equation} \label{inr} I(N, r)=\{(i_1, i_2, \cdots, i_r)\mid 1\le i_j\le N, \forall 1\le j\le r\}.\end{equation}
If $\mathbf i\in I(N, r)$, we write \begin{equation}\label{vi} v_{\mathbf i}=v_{i_1}\otimes v_{i_2}\otimes \cdots\otimes v_{i_r}.\end{equation}

\begin{corollary}\label{fact} Let  $V$ be  the {natural representation} of $\mathbf U_\kappa(\mathfrak g)$.  If $(i_1,i_2),(j_1,j_2)\in I(N,2)$, then
$ v_{j_1}\otimes v_{j_2}\in V^{\otimes 2}$ is involved in  $(v_{i_1}\otimes v_{i_2})\check{R}$  only if $ j_1\leq i_2$ and $ j_2\geq i_1$. \end{corollary}
\begin{proof} The result was given in \cite{Hu} for $\mathbf U_\kappa(\mathfrak {sp}_{2n})$. The other cases  follow from Lemma~\ref{basic2}, immediately.
\end{proof}

\medskip

\begin{definition}~\cite{BM, Mu}\label{bmw-def}   Let $R$ be a commutative ring containing $1$ and invertible elements $\varrho, q$ and $q-q^{-1}$.
The Birman-Murakami-Wenzl algebra
$\mathscr B_r(\varrho, q)$ is the unital
associative $R$-algebra generated by $T_i, E_i,1\le i\le r-1$ satisfying
relations\begin{enumerate} \item
$(T_i-q)(T_i+q^{-1})(T_i-\varrho^{-1})=0$ for $1\le i\le r-1$,
\item $T_{i}T_{i+1}T_i=T_{i+1}T_iT_{i+1}$,  for $1\le i\le r-2$,
\item $ T_iT_j=T_jT_i$,  for $|i-j|>1$,
\item $E_iT_j^{\pm 1} E_i=\varrho^{\pm 1}E_i$,  for $1\le i\le r-1$ and $j=i\pm 1$,
\item $ E_i T_i=T_iE_i=\varrho^{-1} E_i$,  for $1\le i\le r-1$.
\end{enumerate}
where $T_i-T_{i}^{-1} =(q-q^{-1}) ( 1-E_i)$  for $1\le
i\le r-1$.\end{definition}

The  following  results  follow from Definition~\ref{bmw-def}, immediately.
\begin{lemma}\label{anti} Let  $\mathscr B_r(\varrho, q) $ be  defined over $R$. \begin{enumerate} \item There is an $R$-linear anti-involution $ \sigma$ of $\mathscr B_r(\varrho, q) $ fixing $T_i$ and $E_i$, $1\le i\le r-1$.
\item There is an $R$-linear automorphism $\gamma $ of $\mathscr B_r(\varrho, q)$ such that $\gamma(T_i)=T_{r-i}$ and $\gamma(E_i)=E_{r-i}$, $1\le i\le r-1$.\item
Let $\tilde{\sigma}=\sigma\circ\gamma$. Then  $\tilde{\sigma}$  is an $R$-linear anti-involution of  $\mathscr B_r(\varrho, q)$ such that $\tilde{\sigma}(T_i)=T_{r-i}$ and $\tilde{\sigma}(E_i)=E_{r-i}$, $1\le i\le r-1$.\end{enumerate}
\end{lemma}

In this paper, we need  Enyang's result on a basis of $\mathscr B_r (\varrho, q)$ in \cite{Enyang}.
Let $\mathfrak S_r$ be the symmetric group in $r$ letters $\{1, 2, \cdots, r\}$. Then $\mathfrak S_r$ is a Coxeter group with generators $s_1, s_2, \cdots, s_{r-1}$ satisfying usual braid relations together with
$s_i^2=1$, $1\le i\le r-1$.
For each  integer $f$,  $1\le f\le \lfloor\frac r2 \rfloor$,  let $\mathfrak B_f$
be the subgroup of $\mathfrak S_r$ generated by $s_{1}$,
and $s_{2i-2} s_{2i-1}s_{2i-3}s_{2i-2}$,
$2\le i\le f$.  If $f=0$, we set $\mathfrak B_f=1$.   Enyang~\cite{Enyang}  described  $\mathscr D_{f}$, a  complete set of right coset
representatives of $\mathfrak B_f\times \mathfrak S_{r-2f}  $ in
$\mathfrak S_r$, where   $\mathfrak
S_{r-2f}$ is  the subgroup of $\mathfrak S_r$ generated by $s_j$, $2f+1\le j\le r-1$.
For any $w\in \mathfrak S_r$, write $T_w=T_{i_1}T_{i_2}\cdots
T_{i_k}\in \mathscr B_r(\varrho, q)$ if $s_{i_1}\cdots s_{i_k}$ is a reduced expression of
$w$. It is known that  $T_w$ is independent of a reduced
expression of $w$.

\begin{theorem}\cite{Enyang}\label{basis} Suppose that $R$ is a commutative ring containing $1$ and invertible
elements $\varrho, q$ and $q-q^{-1}$.  Then   $S_1
=\{T^*_{d_1}E^f T_w T_{d_2}\mid 0\leq f\leq  \lfloor r/2\rfloor, w\in\mathfrak{S}_{r-2f}, d_1,d_2\in\mathscr D_{ f} \}$ is an $R$-basis of  $ \mathscr B_r(\varrho,q)$,
where $E^f=E_{1}E_{3} \cdots E_{2f-1} $ for $f>0$  and $E^0=1$, and ``$\ast$'' is the $R$-linear anti-involution $\sigma$ on $\mathscr  B_{r}(\varrho, q)$
given in Lemma~\ref{anti}(1).
\end{theorem}

Let $\mathscr D^f$ be the set of
distinguished right coset representatives of $\mathfrak B_f$ in the subgroup $\mathfrak S_{2f}$ of $\mathfrak S_r$ generated by $s_i, 1\le i\le 2f-1$.
It was defined in  \cite{DDH} that
\begin{equation} \label{pf} P_f=\{(i_1,i_2,\cdots, i_{2f})\mid 1\leq i_1<\cdots<i_{2f}\leq r \}.\end{equation}
For each $J\in P_f$, define \begin{equation}\label{dj} d_J=s_{2f,i_{2f}}s_{2f-1,i_{2f-1}}\cdots s_{2,i_2}s_{1,i_1},\end{equation}
where $s_{i, j}=s_{i}s_{i+1, j}$ (resp., $1$) for  $i<j$ (resp., $i=j$) and   $ s_{i, j}=s_{j, i}^{-1}$ if $i>j$.
Then $d_J$ is the  unique element in $\mathscr D_f$
such that $(k)d_J=i_k$, $1\le k\le 2f$.
Further, by  \cite[Lemma~3.8]{DDH},
\begin{equation} \mathscr D_f =\dot \bigcup_{J\in P_f} \mathscr D^f d_J,\end{equation}
where $\dot\cup$ denotes a disjoint union. Following \cite{Hu}, define   $J_0=(r-2f+1,\cdots,r-1,r)\in P_f$ and
  $d_0=s_{2f-2,2f}s_{2f-4,2f}\cdots s_{2,2f}\in
  \mathscr  D^f$.

\begin{lemma}\cite[Lemma~5.12]{Hu}\label{dd}
\begin{enumerate}
\item  For any $d\in \mathscr D_f$, there is a $w\in \mathfrak S_r$, such that $d_0=dw$ and $\ell(d_0)=\ell(d)+\ell(w)$, where $\ell(\ )$ is the length function on $\mathfrak S_r$.
\item For any  $J\in P_f$, there is a $w'\in\mathfrak S_r$, such that $d_{J_0}=d_Jw'$ and $\ell(d_{J_0})=\ell(d_J)+\ell(w')$.
\item For any $d\in \mathscr D_f$ with $d\neq d_0d_{J_0}$, there is a $j$ with $1\leq j<r$, such that $ds_j\in\mathscr D_f$ and $\ell(ds_j)=\ell(d)+1$.
\end{enumerate}
\end{lemma}

In the remaining part of this section, we always assume that
\begin{equation}\label{varrho}
\varrho=
 \begin{cases}-q^{2n+1}, &\quad \text{if $ \mathfrak g=\mathfrak{sp}_{2n}$,}\\
 q^{2n-1}, &\quad \text{if  $\mathfrak g=\mathfrak{so}_{2n}$,}\\
  q^{2n}, &\quad \text{if  $\mathfrak g=\mathfrak{so}_{2n+1}$. }\\
 \end{cases}
\end{equation}
Let  $V$ be  the natural representation of  $\mathbf U_\kappa(\mathfrak g)$  with $\mathfrak g\in \{
\mathfrak{sp}_{2n}, \mathfrak {so}_{2n}, \mathfrak {so}_{2n+1}\}$.  If $\varrho$ is given in \eqref{varrho}, then there is a $\kappa$-algebra homomorphism
\begin{equation} \label{isoo1} \varphi: \mathscr B_r(\varrho,  q )\rightarrow  \text{End}_{\mathbf U_\kappa(\mathfrak g)}(V^{\otimes r})\end{equation} such that
\begin{equation} \varphi(T_i)=1^{\otimes i-1}\otimes \breve{R}\otimes 1\otimes \cdots \otimes 1 \text{ and }  \varphi(E_i)=1^{\otimes i-1} \otimes E\otimes 1\otimes \cdots \otimes 1.\end{equation} We remark that $\varphi$ has been defined in \cite{ha} when $\kappa$ is $\mathbb C(v)$. However, since $V$ contains an $\mathcal A$-lattice which is a left  $\mathbf U(\mathfrak g)$-module, by base change, it can be defined over an arbitrary field $\kappa$.

In the remaining part of this section, all results  for $\mathbf U_\kappa(\mathfrak {sp}_{2n})$ have been proved in \cite{Hu}. The corresponding results for both $\mathbf U_\kappa(\mathfrak{so}_{2n})$ and $\mathbf U_\kappa(\mathfrak{so}_{2n+1})$ can also be proved by arguments in
 \cite{Hu}. For self-contained reason, we give a sketch.

\begin{lemma}\label{ker} (cf.~\cite[Lemma~5.6]{Hu})  Suppose $n\ge r$. Then $\text{ker }\varphi\subseteq \mathscr B_r(\varrho,q)^1$,
where $\mathscr B_r(\varrho,q)^f$ is the two-sided ideal of $\mathscr B_r(\varrho,q)$ generated by $E^f$, $1\leq f\leq [r/2]$.
\end{lemma}
\begin{proof} Recall that $\{v_i\mid 1\le i\le N\}$ is a basis of $V$. Let $v=v_{r}\otimes v_{r-1}\otimes\cdots\otimes v_1\in V^{\otimes r}$.  If $x\in \ker\varphi$,  then $v x=0$. It is proved in \cite[Lemma~5.6]{Hu} that $x\in  \mathscr B_r(\varrho,q)^1$ if  $\mathfrak g=\mathfrak{sp}_{2n}$ and $n\ge r$.
 By \eqref{form},
$\mathscr B_{r}(\varrho, q)$ acts  on $v$ via  the same formula for $\mathfrak{g}\in \{\mathfrak{sp}_{2n}, \mathfrak{so}_{2n}, \mathfrak{so}_{2n+1}\}$.
  So, the results for $\mathfrak{so}_{2n}, \mathfrak{so}_{2n+1}$  follow from similar arguments.
\end{proof}

For $\textbf{i}\in I(N,r)$, let $\ell(v_{\textbf{i}})=\ell(\textbf{i})$, which is  the maximal number of disjoint  pairs $(s,t) $  such that $i_s=(i_t)'$.
When $\mathfrak g=\mathfrak{sp}_{2n}$, $\ell(v_{\textbf{i}})$ is called  the \textsf{ symplectic length} of $\mathbf i$ in \cite{Hu}.
The following result has been given in
\cite[Lemma~5.14]{Hu} for $\mathfrak g=\mathfrak{sp}_{2n}$. In Cases~2--3 of the proof of \cite[Lemma~5.14]{Hu},
 Hu used  \cite[(5.13)]{Hu}   and did not use the explicit
description of  $(v_{i_1}\otimes v_{i_2})\check{R}$. If $\mathfrak g\in\{ \mathfrak{so}_{2n}, \mathfrak {so}_{2n+1}$\},   \cite[(5.13)]{Hu} is still true (see Corollary~\ref{fact}).
So,  arguments in the proof of \cite[Lemma~5.14]{Hu} can be used smoothly to give the proof of the corresponding results for both $\mathfrak {so}_{2n}$ and $\mathfrak{so}_{2n+1}$
as follows\footnote{
We remark that $\tilde w$ in \cite[Lemma~5.14]{Hu} should be read as $ v_{\mathbf j}$ in Lemma~\ref{usefullemma} so that one can get a suitable  induction assumption in Cases~2--3 in the proof of  \cite[Lemma~5.14]{Hu}.}.

\begin{lemma}\label{usefullemma}(cf.\cite[Lemma~5.14]{Hu}) Fix a positive integer $s$ with $1\le s\le f$ and assume that
 $\mathbf i\in I(N, a)$ such that either $1\leq i_j\le  n-f$ or $n'\leq i_j\leq (n-f+s+1)'$ for  each integer $j$ with  $1\leq j\leq a$. Suppose that  $d$ is a distinguish right coset representative of $\mathfrak S_{2s, a}$
in  $\mathfrak S_{2s+a}$, where $\mathfrak S_{2s, a}$
is the subgroup of $\mathfrak S_{2s+a}$ generated by $s_j$ with $j\neq 2s$.  If
 $J=(a+1,a+2,\cdots,a+2s)$,  and $\mathbf j=((n-f+s)',  \cdots, (n-f+2)', (n-f+1)', n-f+1,
 \cdots, n-f+s)$, then
$$(v_{\mathbf i} \otimes v_{\mathbf j})T_{d^{-1}}=q^z\delta_{d,d_J} v_{\mathbf j}\otimes v_{\mathbf i} +\sum_{{\bf u}\in I(N,2s+a)}a_{\bf u}v_{\bf u},$$
for some $z\in\mathbb Z$ such that  $a_{\bf u}\neq0 $ only if $\ell(u_1,\cdots,u_{2s})<s$, and
$x\not\in \{u_1, u_2, \ldots, u_{2s}\}$ for any positive integer $x$ satisfying  either  $(n-f)'\le x\leq 1'$ or $n-f+s+1\leq x\leq  n$.
\end{lemma}

Following \cite{Hu}, let  \begin{equation}\label{if} I_f=\left\{(b_1,\cdots,b_{r-2f})\mid 1\leq b_{r-2f}<\cdots<b_2<b_1\leq n-f\right\}.\end{equation}
The following result is the counterpart of \cite[Lemma~5.15]{Hu}. It can be proved by arguments similar to those in the proof of \cite[Lemma~5.15]{Hu}. The difference is that one needs to use
Lemma~\ref{usefullemma} instead of \cite[Lemma~5.14]{Hu}.

\begin{lemma}\label{u2}(cf.  \cite[Lemma~5.15]{Hu})  Suppose  $v=v_{\bf b}\otimes v_{\bf c}\in V^{\otimes r}$ for some  ${\bf b}\in I_f$ and  ${\bf c}=(n',(n-1)',\cdots,(n-f+1)',n-f+1,\cdots,n-1,n)$\footnote{The element  $v_{\bf c}$ in \cite[Lemma~5.18]{Hu} should be read as current $v_{\bf c} $ so as to be compatible with $v_{\mathbf j}$ in Lemma~\ref{usefullemma}.}.  If
 $w\in \mathscr D_f$ such that  $w\neq d_0d_{J_0}$, then $ (v) T^*_wE^f =0$.
\end{lemma}

For any $v\in V^{\otimes r}$,   let  $\text{ann}(v)=\{x\in\mathscr B_r(\varrho,q) \mid vx=0 \}$.
 The following result, which is the key step in the proof of the injectivity of $\varphi$,  is the counterpart of \cite[Lemma~5.18]{Hu}.
  \begin{lemma}\label{key lemma} (cf.  \cite[Lemma~5.18]{Hu}) Let  $ M$ be  the
  $\kappa$-space spanned by $$S=\{T^*_{d_1}E^fT_\sigma T_{d_2}\mid d_1,d_2\in\mathscr  D_f, d_1\neq d_0d_{J_0}, \sigma\in\mathfrak S_{r-2f} \}.$$ Then   $\mathscr B_r(\varrho,q)^{f}\bigcap  \bigcap_{{\mathbf b}\in I_f}\text{ann}(v_{\mathbf b}\otimes v_{\bf c})=\mathscr B_r(\varrho,q)^{f+1}\oplus M$. \end{lemma}

  \begin{proof} Note that $\ell (v_{\bf b})=0$ for any ${\mathbf b}\in I_f$. So,  $\mathscr B_r(\varrho,q)^{f+1}\subseteq\text{ann}(v_{\bf b}\otimes v_{\bf c})$.
  By Lemma~\ref{u2}, we have the result for "$\supseteq$''. Conversely, for any $x\in\mathscr B_r(\varrho,q)^{f}\bigcap \bigcap_{{\bf b}\in I_f}\text{ann}(v_{\bf b}\otimes v_{\bf c})$, By Theorem~\ref{basis} and Lemma~\ref{u2}, we can  write  $$x=T^*_{d_0d_{J_0}}E^f(\sum_{d\in\mathscr D_f}z_dT_d)+h, $$
  where $z_d=\sum_{\sigma\in\mathfrak S_{r-2f}} a_\sigma T_\sigma$, for  $d\in\mathscr D_f$, $ a_\sigma\in\kappa$ and $h\in \mathscr B_r(\varrho,q)^{f+1}\oplus M$. In order to prove the result for  "$\subseteq$", it suffices to show  $z_d=0$, for each $d\in\mathscr D_f$.
  In \cite{Hu}, Hu proved  $z_{d_0d_{J_0}}=0$ for $\mathfrak g=\mathfrak{sp}_{2n}$. Further, since his proof depends on \cite[(5.13)]{Hu} and does not depend on the explicit description
  of  $(v_{i_1}\otimes v_{i_2})\check{R}$, one can use Corollary~\ref{fact} to replace \cite[(5.13)]{Hu} in Step~1 in  the proof of  \cite[Lemma~5.18]{Hu}. So, $z_{d_0d_{J_0}}=0$.
 By Lemma~\ref{dd}, $ d_0d_{J_0}$ is the maximal element of $\mathscr D_f$ with respect to  the Bruhat order. Mimicking arguments in the  proof of Step~2 of  \cite[Lemma~5.18]{Hu}, i.e. by induction on $\ell(d)$ for  $d\in\mathscr D_f$, one can verify $z_{d}=0$ for $ d\neq d_0d_{J_0}$
   \end{proof}

 The following result, which is \cite[Theorem~5.19]{Hu} for $\mathfrak g=\mathfrak{sp}_{2n}$, can be verified via arguments on induction of $\ell(d)$ in  the proof of  \cite[Theorem~5.19]{Hu}.
\begin{lemma}(cf. \cite[Theorem~5.19] {Hu}) \label{kkey}
$\ker\varphi\subseteq \mathscr B_r(\varrho,q)^{f+1}$ if $\ker\varphi\subseteq \mathscr B_r(\varrho,q)^{f}$.
\end{lemma}

\begin{theorem}\label{isom}  Let  $V$ be  the natural representation of  $\mathbf U_\kappa(\mathfrak g)$  with $\mathfrak g\in \{
\mathfrak{sp}_{2n}, \mathfrak {so}_{2n}, \mathfrak {so}_{2n+1}\}$.   Then $\varphi$ defined in \eqref{isoo1}  is a $\kappa$-algebra isomorphism
if \begin{enumerate}\item  $\mathfrak g=\mathfrak {sp}_{2n}$ with $n\ge r$, \item   $\mathfrak g\in \{ \mathfrak {so}_{2n}, \mathfrak {so}_{2n+1}\}$ with  $n>r$.\end{enumerate}
\end{theorem}

\begin{proof}
 We remark that (1) has been proved  in \cite{Hu}.    If $\mathfrak g\in \{\mathfrak{so}_{2n+1}, \mathfrak{so}_{2n}\}$, $\varphi$ is well-defined over $\kappa$ (in fact, over $R$).
 Further, by Lemma~\ref{ker} and  Lemma~\ref{kkey},  $\ker\varphi\in  \mathscr B_r(\varrho,q)^{f}$ for all positive integers $f$, forcing  $\ker\varphi=0$.
In order to complete proof, it is enough to show that the dimensions of $\mathscr B_{r}(\varrho, q)$ and $ \text{End}_{\mathbf U_\kappa (\mathfrak g)}(V^{\otimes r})$ are the same.
It was defined in  \cite[Definition~2.1]{Ander} that a tilting module for $\mathbf U_\kappa(\mathfrak g)$  is  a finite dimensional left $\mathbf U_\kappa(\mathfrak g)$-module which has a  Weyl-filtration  and  a co-Weyl filtration.
Since $V=\Delta(\epsilon_1)$, the Weyl module with highest weight $\epsilon_1$,  and $V\cong V^*$, $V$ is a tilting module for  $\mathbf U_\kappa(\mathfrak g)$ and so is $V^{\otimes r}$. By Lemma~5.1 in \cite{AGZ},  the dimension of $ \text{End}_{\mathbf U_\kappa(\mathfrak g)}(V^{\otimes r})$ is independent of $\kappa$. In particular,  we  assume $\kappa=\mathbb C(v)$ where $v$ is an indeterminate. In this case, $V^{\otimes r}$ is completely reducible.  By \cite[(5.5)]{LR} and  Enyang's construction of Jucys-Murphy basis of $\mathscr B_r(\varrho, q)$  in \cite{Enyang},   the dimension of $ \text{End}_{\mathbf U_\kappa(\mathfrak g)}(V^{\otimes r})$ is equal to that of  $\mathscr B_{r}(\varrho, q )$. So, $\varphi$ is surjective.  \end{proof}

\section{ An invariant  form on $V^{\otimes r}$}
 In this section, we always  assume  that  $\kappa$ is  a field containing $q$  (resp., $q^{1/2}$ if $\mathfrak g=\mathfrak{so}_{2n+1}$) such that $q^2\neq 1$. Let  $ V$  be   the natural representation of
  $\mathbf U_\kappa(\mathfrak g)$, with
 $\mathfrak g\in \{\mathfrak{so}_{2n+1}, \mathfrak{sp}_{2n},
 \mathfrak{so}_{2n}\} $.   The aim of this section is to prove  that  $V^{\otimes r}$  is self-dual
 as $(\mathbf U_\kappa(\mathfrak g), \mathscr B_r(\varrho, q)$)-module if $\varrho$ is given in \eqref{varrho}.

First, we consider $\mathfrak g=\mathfrak{so}_{2n+1}$. For any  $\mathbf i\in I(2n+1, r)$, define $ {\mathbf i}'\in I(2n+1, r)$ such that $ {\mathbf i}'=( i_r',  i_{r-1}', \
\cdots,  i_1')$ if $\mathbf i =(i_1, \cdots, i_r)$, where  $ i'=2n+2-i$,  and  ${ i}''=i$,  $1\le i\le n$.

\begin{lemma}\label{bi}   For any positive integer $r$,
define the $\kappa$-bilinear form $\langle \ \   ,\ \ \rangle: V^{\otimes r}\times V^{\otimes r}\rightarrow \kappa$    such that
\begin{equation} \label{bi12} \langle v_{\mathbf i} ,v_{\mathbf j} \rangle=q^{-\rho_{\mathbf i} }\delta_{\mathbf i, {\mathbf j}'}, \quad \text{for $ \mathbf i, \mathbf j\in I(2n+1,r)$,}\end{equation}
 where $\rho_{\mathbf i} =\sum_{k=1}^r\rho_{i_k}$, and $\rho$ is given in \eqref{rho}.
\begin{enumerate}\item The bilinear form  $\langle \ \   ,\ \ \rangle $  is   non-degenerate.
\item $\langle a v_{\mathbf i} ,v_{\mathbf j} \rangle=\langle v_{\mathbf i}  , S(a)v_{\mathbf j} \rangle $, $a\in \mathbf U_{\kappa}(\mathfrak {so}_{2n+1})$,  $\mathbf i, \mathbf j\in I(2n+1, r)$, where $S$ is the antipode of  $\mathbf U_\kappa(\mathfrak {so}_{2n+1})$ given in \eqref{coalg}.
\item $ \langle  v_{\mathbf i} b,v_{\mathbf j} \rangle=\langle v_{\mathbf i} , v_{\mathbf j}  \tilde{\sigma}(b)\rangle$,  $b\in\mathscr B_r(\varrho, q)$,  $\mathbf i, \mathbf j\in I(2n+1, r)$,  where $\tilde{\sigma}$ is the anti-involution on $\mathscr B_{r}(\varrho, q)$
given in Lemma~\ref{anti}(3).
\end{enumerate}
\end{lemma}
\begin{proof} We remark that (1) follows from  \eqref{bi12}, immediately. Let $V^*$ be the $\kappa$-linear dual of $V$. Then $V\cong V^*$ as left $\mathbf U_\kappa(\mathfrak{so}_{2n+1})$-modules  and the corresponding isomorphism $\varphi$  satisfies
  \begin{equation} \label{varphi1}  \varphi (v_i)=q^{-\rho_i}v_{{i'}}^*,\quad  1\le i\le 2n+1, \end{equation}
 where  $\{v_i^\ast\mid  1\le i\le 2n+1\}$ is the dual basis of a basis  $\{v_i\mid 1\le i\le 2n+1\}$ of $V$.
  By  Proposition~111.5.2 in \cite{CK},  $ M^*\otimes N^*\cong (N\otimes M)^*$
for  any finite dimensional $\mathbf U_{\kappa}(\mathfrak {so}_{2n+1}) $-modules $M$ and $N$.  So $(V^*)^{\otimes r}\cong  (V^{\otimes r})^*$ and the corresponding isomorphism $\Psi: (V^*)^{\otimes r}\rightarrow  (V^{\otimes r})^*$ satisfies
\begin{equation} \label{psi}  \Psi(v_{i_1}^*\otimes \cdots\otimes v_{i_r}^*)=(v_{i_r}\otimes \cdots\otimes v_{i_1})^*, \quad \mathbf i\in I(2n+1, r). \end{equation}
 Thus \begin{equation}\label{Phi}  \Phi:  V^{\otimes r} \cong (V^{\otimes r})^* \end{equation}  as
 left $\mathbf U_\kappa(\mathfrak {so}_{2n+1}) $-modules where $\Phi= \Psi\circ \varphi^{\otimes r}$.
It is routine to check that   \begin{equation}\label{Phi1} \Phi(v_{\mathbf i})(v_{\mathbf j} )=\langle v_{\mathbf i}, v_{\mathbf j} \rangle, \quad \forall \mathbf i, \mathbf j\in I(2n+1, r).\end{equation} Now, (2) follows since it is equivalent to saying that   $\Phi$ is a left $\mathbf U_\kappa(\mathfrak {so}_{2n+1})$-homomorphism.
By  Definition~\ref{bmw-def},  $\mathscr B_r(\varrho, q)$ can be generated by $T_i^{\pm 1}$, $1\le i\le r-1$.
 In order to prove (3), by  \eqref{bi12},  it suffices to verify \begin{equation} \label{bi11} \langle  v_{\mathbf i} ,v_{\mathbf j}  T_1\rangle=\langle  v_{\mathbf i} T_1,v_{\mathbf j} \rangle\end{equation} for $r=2$. If so, we have  $\langle  v_{\mathbf i} ,v_{\mathbf j}  T_1^{-1}\rangle=\langle  v_{\mathbf i}T_1^{-1} T_1,v_{\mathbf j} T_1^{-1} \rangle=\langle  v_{\mathbf i}T_1^{-1} ,v_{\mathbf j}\rangle $, proving (3).

By  \eqref{form}, it is easy to check \eqref{bi11}  if $i_1\neq  i_2'$. Assume $i_1= i_2'$ and
 write  $\delta=q-q^{-1}$.   If $\mathbf i=(i_1, i_2)=(n+1, n+1)$, then
$$\langle  v_{\mathbf i} ,v_{\mathbf j}  T_1\rangle=\langle  v_{\mathbf i} T_1,v_{\mathbf j} \rangle=
\begin{cases}1, \quad & j_1=j_2=n+1,\\
-\delta q^{-\rho_{j_1}}, \quad& j_2= j_1'>n+1,\\ 0 , & \text{ otherwise.} \\
\end{cases} $$
Suppose $\mathbf i\neq (n+1, n+1)$.
If $i_1>i_2$, then
$$\langle  v_{\mathbf i} ,v_{\mathbf j}  T_1\rangle=\langle  v_{\mathbf i} T_1,v_{\mathbf j} \rangle=
\begin{cases}q^{-1}, \quad & (j_1,j_2)=(i_2,i_1),\\
-\delta q^{\rho_{j_2}-\rho_{i_1}}, & \text{ $j_2= {j_1}'>i_1$, } \\ 0 ,  & \text{ otherwise. }  \\
\end{cases} $$
If $i_1<i_2$ and  $\mathbf j=\mathbf i$, then  \eqref{bi11} follows from  \eqref{bi12}. If $i_1<i_2$ and  $\mathbf j\neq \mathbf i$,
$$ \langle  v_{\mathbf i} ,v_{\mathbf j}  T_1\rangle=\langle  v_{\mathbf i} T_1,v_{\mathbf j} \rangle=
\begin{cases}q^{-1},  & (j_1,j_2)=(i_2,i_1),\\
-\delta q^{\rho_{j_2}-\rho_{i_1}}, & \text{ $i_2\neq j_2=j_1'>i_1$,}\\
0 ,\quad & \text{ otherwise.}
\end{cases} $$
In any case, we have \eqref{bi11}, proving (3).  \end{proof}

For any right  $\mathscr B_r(\varrho, q)$-module $M$,   $M^* $ is a  right  $\mathscr B_r(\varrho, q)$-module such that
\begin{equation}\label{dualm1}  (\phi b) (x)=\phi(x\tilde{\sigma}(b)),  \forall \phi\in M^*, b\in \mathscr B_r(\varrho, q), x\in M, \end{equation} where $\tilde{\sigma}$ is the anti-involution on $\mathscr B_{r}(\varrho, q)$
given in Lemma~\ref{anti}(3).

\begin{coro}\label{duforb}  As $(\mathbf U_{\kappa}(\mathfrak {so}_{2n+1}),  \mathscr B_r(\varrho, q))$--bimodules, $ V^{\otimes r}\cong (V^{\otimes r})^* $  where
$\varrho$ is given in \eqref{varrho}.
\end{coro}
\begin{proof} By Lemma~\ref{bi}(2)--(3) and  \eqref{Phi1}, the $\Phi$ given in \eqref{Phi} is the required isomorphism.\end{proof}

Now, we assume  $\mathfrak g\in \{\mathfrak{sp}_{2n}, \mathfrak{so}_{2n}\}$.
Recall that  $\tau :  \mathbf U_{\kappa }(\mathfrak  g) \rightarrow \mathbf U_\kappa(\mathfrak  g)  $ is an anti-automorphism such that
\begin{equation} \label{tau1} \tau(k_i)=k_i, \tau(e_i)=f_i \text{ and }  \tau(f_i)=e_i, \ \ \text{$1\le i\le n$.}\end{equation}
For any left $ \mathbf U_\kappa(\mathfrak g)$-module $N$, let $N^{\circ}$  be the left $ \mathbf U_\kappa (\mathfrak g)$-module  such that $N^{\circ}=N^*$  as $\kappa$-vector spaces, and the action is given by
\begin{equation} \label{contramod}
(u\phi)(x)=\phi(\tau(u)x), \forall x\in N, u\in \mathbf U_{\kappa}(\mathfrak g),\phi\in N^*.\end{equation}
Let $\varrho\in \kappa$ be given in \eqref{varrho}. For any right $\mathscr B_r(\varrho, q)$-module $M$, let  $M^{\circ}$ be the right  $\mathscr B_r(\varrho, q)$-module  such that $M^{\circ}=M^*$ as $\kappa$-vector spaces, and the action is given by
\begin{equation} \label{contramod1} (\phi b)(y)=\phi(y\sigma(b)), \forall  y\in M, b\in\mathscr B_r(\varrho, q),\phi\in M^*,\end{equation}
where $\sigma $ is the anti-involution on $\mathscr B_{r}(\varrho, q)$  given in Lemma~\ref{anti}.

\begin{lemma}\label{contra}  For any positive integer $r$,
let $\langle\  ,\ \rangle: V^{\otimes r}\times V^{\otimes r}\rightarrow \kappa $ be the bilinear form such that
\begin{equation}\label{mf1}\langle v_{\mathbf i}, v_{\mathbf j} \rangle=q^{\beta_{\mathbf i} }\delta_{\mathbf i,\mathbf j}, \quad \forall  \mathbf i, \mathbf j\in I(2n,r),
\end{equation}
 where $\beta(\mathbf i)=\sharp\{i_k\neq i_j\mid j\neq k, i_j\neq  i_k' \}+2\sharp\{i_k= i_j' \}$. Then
\begin{enumerate}
\item $\langle \ , \ \rangle $  is symmetric and non-degenerate.
\item
$ \langle uv,w\rangle =\langle v,\tau(u)w\rangle$,  $\forall u\in\mathbf U_{\kappa}(\mathfrak g)$ and  $ v,w\in V^{\otimes r}$,  where $\tau$ is the anti-automorphism of $\mathbf U_\kappa(\mathfrak g)$ given in \eqref{tau1}.
\item   $\langle vb,w\rangle =\langle v, w\sigma(b)\rangle $, $\forall  b\in \mathscr B_r(\varrho, q)$ and $ v, w\in V^{\otimes r}$, where $\sigma$ is the anti-involution defined in Lemma~\ref{anti}.
    \end{enumerate}
\end{lemma}
\begin{proof} (1) follows from \eqref{mf1}, immediately. In order to prove (2), it suffices to verify \begin{equation} \label{cont1} \langle uv,w\rangle =\langle v,\tau(u)w\rangle  \end{equation}
for all  $v,w\in V^{\otimes r}$ and $u\in \{e_i, f_i, k_i\mid 1\le i\le n\}$.
It is easy to check \eqref{cont1} if $u=k_i$. Since $\langle \ , \ \rangle $ is symmetric, it remains to check \eqref{cont1} when
$u=e_i$, $1\le i\le n$.

First, we assume $i\neq n$.
Suppose $v=v_{\mathbf i} $ and $w=v_{\mathbf j}$ for $\mathbf i, \mathbf j\in I(2n, r)$. Then $\langle e_i v_{\mathbf i}, v_{\mathbf j}\rangle =0$ unless
there is a  $k, 1\leq k\leq r$ such that  $(i_k,j_k)\in\{(i+1,i), ( i', (i+1)')\} $ and $j_l=i_l$ for all $l\neq k$. In the later case,
let $\alpha_{a}$ (resp., $\gamma_a $) be  the numbers of $ a$ appearing  in $(i_1,\cdots,i_{k-1}) $ (resp., $(i_{k+1},\cdots,i_r)$ ).
Then
\begin{equation}\label{eb}
\beta(\mathbf i)=\beta(\mathbf j)+\alpha_i+\alpha_{ (i+1)'}-\alpha_{ i'}-\alpha_{i+1}+\gamma_i+\gamma_{ (i+1)'}-\gamma_{ i'}-\gamma_{i+1}.\end{equation}
It is routine to check
\begin{equation} \label{eq1} \langle e_iv_{\mathbf i},v_{\mathbf j} \rangle =(-1)^{\delta_{i_k,  i'} } q^{\alpha_i+\alpha_{({i+1})'}-\alpha_{i'}-\alpha_{i+1}+\beta(\mathbf j)}, \end{equation} and
\begin{equation}\label{eq2} \langle v_{\mathbf i},f_iv_{\mathbf j}\rangle = (-1)^{\delta_{i_k,  i'} } q^{-\gamma_i-\gamma_{({i+1})'}+\gamma_{i'}+\gamma_{i+1}+\beta(\mathbf i)}.\end{equation}
 By \eqref{eb}--\eqref{eq2},  $\langle e_iv_{\mathbf i} ,v_{\mathbf j} \rangle = \langle v_{\mathbf i} ,f_iv_{\mathbf j} \rangle $ if  $\langle e_iv_{\mathbf i} ,v_{\mathbf j} \rangle \neq 0$. Finally, it is easy to check   $\langle e_iv_{\mathbf i} ,v_{\mathbf j}\rangle =0$ if and only if
  $\langle v_{\mathbf i} ,f_iv_{\mathbf j} \rangle =0$.

Suppose $i=n$. We have $\langle e_n v_{\mathbf i}, v_{\mathbf j}\rangle =0$ unless one of two conditions holds:
\begin{enumerate} \item [(a)]  $j_l=i_l$ unless $j=k$ for some  $k, 1\leq k\leq r$  and  $(i_k,j_k)=(n', n)$  provided $\mathfrak g=\mathfrak{sp}_{2n}$;
\item [(b)]  $j_l=i_l$ unless   $l= k$ for some    $k, 1\leq k\leq r$  and  $(i_k,j_k)\in\{(n', n-1), (({n-1})', n)\}$  provided $\mathfrak g=\mathfrak{s0}_{2n}$.\end{enumerate}
In case (a),  $\beta(\mathbf i)=\beta(\mathbf j)+2\alpha_n-2\alpha_{{n'}}+2\gamma_n-2\gamma_{{n'}}$,  and  hence
$$\langle e_n v_{\mathbf i} ,v_{\mathbf j} \rangle =q^{2\alpha_n-2\alpha_{{n'}}+\beta(\mathbf j)}=q^{-2\gamma_n+2\gamma_{{n'}}+\beta(\mathbf i)}=(v_{\mathbf i} ,f_n v_{\mathbf j} ).$$
In case (b), $ \beta(\mathbf i)  -\gamma_n-\gamma_{n-1}+\gamma_{{n'}}+\gamma_{(n-1)'}   =\beta(\mathbf j)+\alpha_n+\alpha_{n-1}-\alpha_{{n'}}-\alpha_{(n-1)'}=a$,
and $$\langle e_n v_{\mathbf i} ,v_{\mathbf j} \rangle =(-1)^\epsilon q^{a}
=\langle v_{\mathbf i} ,f_n v_{\mathbf j} \rangle ,$$
where $\epsilon=0$ (resp., $1$) if $i_k= n'$ (resp., $i_k=({n-1})'$). In any case, we have \eqref{cont1} if
$\langle e_n v_{\mathbf i} ,v_{\mathbf j} \rangle \neq 0$. Finally,  it is easy to see that
$\langle e_n v_{\mathbf i} ,v_{\mathbf j} \rangle =0$ if and only if  $ \langle v_{\mathbf i} , f_nv_{\mathbf j} \rangle =0$. This completes the proof of (2).

In order to verify  (3), it suffices to assume $v=v_{\mathbf i} $, $w=v_{\mathbf j}$ and  $b=T_k$, $\forall \mathbf i, \mathbf j\in I(2n, r)$  and $1 \leq k\leq r-1$.
We assume  $i_l=j_l$, for $l\neq k,k+1$. Otherwise, $\langle v_{\mathbf i} T_k,v_{\mathbf j}\rangle =\langle v_{\mathbf i},v_{\mathbf j} T_k\rangle  =0$.
By Lemma~\ref{basic2}, $\langle  v_{\mathbf i} T_k,v_{\mathbf j} \rangle = \langle v_{\mathbf i} , v_{\mathbf j} T_k\rangle $ if
 $i_k\neq {i}_{k+1}'$.  In the remaining,   we assume $i_k={i}_{k+1}'$. In particular, $i_k\neq i_{k+1}$.
  Write $\delta=q-q^{-1}$.
 If  $i_k>i_{k+1}$, then
\begin{equation} \label{case1} \langle  v_{\mathbf i} T_k,v_{\mathbf j} \rangle=
\begin{cases}q^{-1+\beta(\mathbf j)},  & (j_k,j_{k+1})=(i_{k+1},i_k),\\
-\delta q^{\rho_{j_{k+1}}-\rho_{i_k}+\beta(\mathbf j)}\varepsilon_{j_{k+1}}\varepsilon_{i_k}, & \text{ $j_{k+1}= j_k'>i_k$,}\\
0 , & \text{ otherwise.} \\
\end{cases} \end{equation}
and
\begin{equation} \label{case2} \langle  v_{\mathbf i} ,v_{\mathbf j} T_k\rangle=
\begin{cases}q^{-1+\beta(\mathbf i)},  & (j_k,j_{k+1})=(i_{k+1},i_k),\\
-\delta q^{\rho_{ i_k'}-\rho_{ j_{k+1}'}+\beta(\mathbf i)}\varepsilon_{j_{k+1}'}\varepsilon_{i_k'}, & \text{ $j_{k+1}= j_k'>i_k$,}\\
0 , & \text{ otherwise.} \\
\end{cases} \end{equation}
If  $i_k<i_{k+1}$, by (1),   we can assume $\mathbf i\neq \mathbf j$ without loss of generality. We have
\begin{equation} \label{case3}  \langle  v_{\mathbf i} T_k,v_{\mathbf j} \rangle=
\begin{cases}q^{-1+\beta(\mathbf j)},  & \text{ $(j_k,j_{k+1})=(i_{k+1},i_k)$,} \\
-\delta q^{\rho_{j_{k+1}}-\rho_{i_k}+\beta(\mathbf j )}\varepsilon_{j_{k+1}}\varepsilon_{i_k}, & \text{ $ i_{k+1}' \neq j_{k+1}= j_k'>i_k$, }\\
0 , & \text{ otherwise.}
\end{cases} \end{equation}
and
\begin{equation}\label{case4} \langle  v_{\mathbf i},   v_{\mathbf j} T_k\rangle=
\begin{cases}q^{-1+\beta(\mathbf i)},  & \text{ $(j_k,j_{k+1})=(i_{k+1},i_k)$,} \\
-\delta q^{\rho_{i_k'}-\rho_{ j_{k+1}'}+\beta(\mathbf i)}\varepsilon_{j_{k+1}'}\varepsilon_{ i_k'}, & \text{ $ i_{k+1}' \neq j_{k+1}= j_k'>i_k$, }\\
0 , & \text{ otherwise.}
\end{cases} \end{equation}
So, $\langle  v_{\mathbf i} T_k,v_{\mathbf j} \rangle=0$ if and only if  $\langle v_{\mathbf i} ,v_{\mathbf j} T_k\rangle =0 $. Further, if  $\langle v_{\mathbf i} ,v_{\mathbf j} T_k\rangle \neq 0 $, then
 $ \beta(\mathbf i )=\beta(\mathbf j)$ and hence  $\langle  v_{\mathbf i} T_k,v_{\mathbf j} \rangle=\langle v_{\mathbf i} ,v_{\mathbf j} T_k\rangle $ by  \eqref{case1}--\eqref{case4}, \eqref{rho} and the definition of $\varepsilon_i$ in Corollary~\ref{1dim}.
\end{proof}

\begin{coro}\label{cond} Suppose $\mathfrak g\in \{\mathfrak{so}_{2n}, \mathfrak{sp}_{2n}\}$.  As $(\mathbf U_{\kappa }(\mathfrak g),  \mathscr B_r(\varrho, q))$-bimodules,
$V^{\otimes r}\cong (V^{\otimes r})^\circ$, where $\varrho$ is given in \eqref{varrho}.
\end{coro}
\begin{proof} Let $\circ: V^{\otimes r}\rightarrow(V^{\otimes r})^\circ$ be $\kappa$-linear map such that
\begin{equation} \label{duall} x^{\circ} (y)=\langle x , y\rangle,  \ \forall x,y\in V^{\otimes r},\end{equation}  where  $ \langle \ ,\ \rangle $ is  given in \eqref{mf1}. By  Lemma~\ref{contra}, $\circ$ is the  required isomorphism.
\end{proof}

\section{ Representations of Birman-Murakami-Wenzl algebras }
In this section, we assume that  $\mathscr B_{r}(\varrho, q)$ is defined over $\kappa$, where $\kappa$ is a field containing non-zero $\varrho$ and $q$ such that $q^2\neq 1$.  The aim of  this section is to  establish a relationship between decomposition numbers of $\mathscr B_{r}(\varrho, q)$  and the multiplicities of Weyl modules in certain indecomposable tilting modules for $\mathbf U_\kappa(\mathfrak g)$  over  $\kappa$, where  $\varrho$
is given in \eqref{varrho} and $\mathfrak g\in \{\mathfrak {so}_{2n+1}, \mathfrak {so}_{2n}, \mathfrak {sp}_{2n}\}$.
We start by recalling some of combinatorics.

Recall that a composition  $\lambda$  of $r$ with at most $n$ parts  is a sequence of non-negative integers $(\lambda_1, \lambda_2, \cdots, \lambda_n ) $ such that $\sum_{i=1}^n \lambda_i=r$. If  $\lambda_i\ge \lambda_{i+1}$ for all possible $i$'s, then $\lambda$ is called a partition. Let $\lambda'$ be the conjugate of $\lambda$.
Then $\lambda'_k=\sharp \{j\mid \lambda_j\ge k\}$.  Let  $\Lambda(n, r)$ (resp., $\Lambda^+(n, r)$)  be the set of all compositions (resp., partitions)  of $r$
with at most $n$ parts. We also use $\Lambda^+(r)$ to denote the set of all partitions of $r$.
For  any $\lambda\in \Lambda^+(r)$,  let $[\lambda]$ be the Young diagram which  is a
collection of boxes (or nodes) arranged in left-justified rows with $\lambda_i$
boxes in the $i$th row of $[\lambda]$.
A $\lambda$-tableau $\s$ is obtained by inserting $i, 1\le i\le
r$ into $[\lambda]$ without repetition.  A $\lambda$-tableau $\s$ is standard if the entries in $\s$ are
increasing both from left to right in each row and from top to the
bottom in each column. Let $\Std(\lambda)$ be the set of all
standard $\lambda$-tableaux.  The symmetric group
$\mathfrak S_r$ acts on $\s$ by permuting its entries.
 Let $\t^\lambda$ (resp.,
$\t_{\lambda}$) be the $\lambda$-tableau obtained from the Young
diagram $[\lambda]$ by adding $1, 2, \cdots, n$ from left to right
along the rows (resp., from top to bottom down the columns). For
example, if $\lambda=(4,3,1)$, then
\begin{equation}\label{tla}
\t^{\lambda}=\young(1234,567,8), \quad \text{ and }
\t_{\lambda}=\young(1468,257,3).\end{equation}
Write  $w=d(\s)$ if $\t^\lambda w=\s$. Then  $d(\s)$ is
uniquely determined by $\s$. In particular, we denote  $d(\t_\lambda)$ by $w_\lambda$.

Let $\mathscr H_r$ be the Hecke algebra associated to the symmetric group $\mathfrak S_r$.  By definition, $\mathscr H_r$ is a unital  associative  $\mathbb Z[q, q^{-1}]$-algebra generated by $g_i, 1\le i\le {r-1}$ satisfying  relations
\begin{enumerate} \item
$(g_i-q)(g_i+q^{-1})=0$, $1\le i\le r-1$,
\item $g_{i}g_{i+1}g_i=g_{i+1}g_ig_{i+1}$,  $1\le i\le r-2$,
\item $ g_ig_j=g_jg_i$,  $|i-j|>1$.
\end{enumerate}
Let  $I$ be the two-sided ideal of $\mathscr B_{r}(\varrho, q)$ generated by $E_1$. By Definition~\ref{bmw-def},  \begin{equation}\label{heckeb} \mathscr H_r\cong \mathscr B_{r}(\varrho, q)/I.\end{equation}
For any $w\in \mathfrak S_r$, write $g_w=g_{i_1}g_{i_2}\cdots
g_{i_k}$ if $s_{i_1}\cdots s_{i_k}$ is a reduced expression of
$w$. It is known that $k$,  the length of $w$,  is unique  although a reduced expression of $w$ is not unique in general. For each partition  $\lambda$ of $r$,  let
 \begin{equation} \label{mnla} \m_\lambda=\sum_{w\in \mathfrak S_\lambda}
q^{\ell(w)}g_ w, \text{ and } \n_\lambda=\sum_{w\in \mathfrak S_\lambda}
(-q)^{-\ell(w)}g_ w, \end{equation} where $\mathfrak S_\lambda$ is the  Young subgroup of $\mathfrak S_r$ with respect to $\lambda$. For any $\s, \t\in \Std(\lambda)$,
let $$\m_{\s\t}=g_{d(\s)}^* \m_\lambda g_{d(\t)},\quad   \n_{\s\t}=g_{d(\s)}^* \n_\lambda g_{d(\t)},$$
   where $\ast$ is the anti-involution on $\mathscr H_r$ such that $g_i^*=g_i$, $1\le i\le r-1$.

If  $\lambda \in \Lambda^+(n, r)$, we define
\begin{equation}\label{ilambda}  \mathbf i_\lambda=(1^{\lambda_1} ,2^{\lambda_2} ,\cdots, n^{\lambda_n})\in I(n,r), \end{equation}
where $I(n, r)$ is defined in \eqref{inr}.
The following result is  a special case of
\cite[Theorem~4.13]{RSong1}.

\begin{lemma}\label{hightest of a} Let $V$ be the natural representation of $\mathbf U_\kappa(\mathfrak {sl}_n)$.
For any $\t\in \Std(\lambda' )$ with $\lambda\in \Lambda^+(n, r)$, let $v_{\lambda,\t}=v_{\mathbf i_\lambda}g_{w_\lambda}\n_{\lambda'}g_{d(\t)}$. If $n\ge r$, then
  $\{v_{\lambda,\t}\mid \t\in\Std(\lambda')\} $ is a basis of $\kappa$-space consisting of all highest weight vectors of $V^{\otimes r}$ with  weight $\sum_{i=1}^{n} \lambda_i\epsilon_i-\frac r n\sum_{i=1}^{n} \epsilon_{i}$.
\end{lemma}

 Let $\Lambda_r=\{(f, \lambda)\mid 0\le f\le \lfloor r/2\rfloor, \lambda\in \Lambda^+(r-2f)\}$.
  For any non-negative  integer $f\le \lfloor r/2\rfloor$, let $\mathscr B_{r-2f}(\varrho, q)$ (resp., $\mathscr H_{r-2f}$) be generated by $T_i$ and $E_i$ (resp., $g_i$), $2f+1\le i\le r-1$.
  In Theorem~\ref{cell12},  $\n_{\s\t}$ is the element in $ \mathscr B_{r-2f}(\varrho, q)$, which is obtained from that of $\mathscr H_{r-2f}$
 by using $T_w$ instead of $g_w$.

\begin{theorem}\cite{Enyang}\label{cell12}  Let $\mathscr B_r(\varrho, q)$ be the Birman-Murakami-Wenzl algebra over a commutative ring  $R$ containing $1$ and invertible elements $\varrho$, $q$ and  $q-q^{-1}$. Let
 $$S=\left\{T^*_{d_1}E^f \n_{\s\t}T_{d_2}\mid (f, \lambda)\in \Lambda_r, \s,\t\in\Std(\lambda),
 d_1,d_2\in\mathscr D_f\right\},$$ where $E^f=E_1E_3\cdots E_{2f-1}$ for $f>0$ and $E^0=1$.
 \begin{enumerate}\item  $S$ is a   cellular basis of $ \mathscr B_r(\varrho,q)$ over $R$ in the sense of \cite{GL},
\item $\gamma (S)$ is another cellular basis of $ \mathscr B_r(\varrho,q)$ over $R$, where $\gamma$ is the automorphism of $\mathscr B_r(\varrho, q)$ defined in Lemma~\ref{anti}.\end{enumerate}
\end{theorem}
In fact,  Theorem~\ref{cell12} has been given in \cite{Enyang}  if one uses indexed representations instead of signed representations for Hecke algebras.   By standard results on the representation theory on cellular algebras in \cite{GL}, for each pair  $(f, \lambda)\in \Lambda_r$,  we have  right cell modules  $C(f, \lambda)$ (resp., $\tilde C(f, \lambda)$) of $\mathscr B_{r}(\varrho, q)$   with respect to the cellular bases  of   $\mathscr B_{r}(\varrho, q)$  in
Theorem~\ref{cell12}(1) (resp., (2)). Further,   there is an invariant form $\phi_{f, \lambda}$  on  $C(f, \lambda)$ (resp., $\tilde C(f, \lambda)$).
 Let $\text{rad} \phi_{f, \lambda}$ be the radical with respect to the invariant form on $C(f, \lambda)$ (resp., $\tilde C(f, \lambda)$). The corresponding quotient $ C(f, \lambda)/\text{rad}\phi_{f, \lambda}$ (resp.,
  $ \tilde C(f, \lambda)/\text{rad}\phi_{f, \lambda}$) will be denoted by $D^{f, \lambda}$ (resp., $ \tilde D^{f, \lambda}$).

 Recall that $e$ is the order of $q^2$. A partition $\lambda$ of $r$ is called $e$-restricted if  $\lambda_i-\lambda_{i+1}<e$ for all possible $i$. If $\lambda'$ is $e$-restricted, then $\lambda$ is called $e$-regular.
  It is proved in \cite{Xi} that $D^{f, \lambda}\neq 0$ if and only if $\lambda$ is $e$-restricted and $f\neq r/2$ if $r$ is even and $\varrho^2=1$. By Theorem~\ref{cell12}(2), similar result holds for  $ \tilde D^{f, \lambda}$.
   Let $P(f, \lambda)$ (resp., $\tilde P(f, \lambda)$) be the projective cover of $D^{f, \lambda}$ (resp., $ \tilde D^{f, \lambda}$).

   The multiplicities of simple $\mathscr B_r(\varrho, q)$-modules $D^{f, \lambda}$ in cell modules  $C(\ell, \mu)$ will be called decomposition numbers of $\mathscr B_r(\varrho, q)$ if $\varrho\neq q^{2n}$ for some $n\in \mathbb N$. When $\varrho= q^{2n}$, we use $\tilde C(\ell, \mu$) and   $ \tilde D^{f, \lambda}$ instead of $C(\ell, \mu)$ and $D^{f, \lambda}$ to define decomposition numbers of $\mathscr B_r(\varrho, q)$.
For any $(f, \lambda)\in \Lambda_r$, define  \begin{equation} \label{hwv12} v_{\lambda}=\underset {2f} {\underbrace{ v_1\otimes v_{{1'}}\otimes \cdots \otimes v_1\otimes v_{{1'}}}} \otimes v_{\mathbf i_\lambda}.\end{equation} In Proposition~\ref{bmwab}, we use $T_w\in\mathscr B_{r}(\varrho, q) $ instead of $g_w\in \mathscr H_r$  in  \eqref{mnla} so as to get corresponding $\m_\lambda$ and $\n_\lambda$ in $\mathscr B_{r}(\varrho, q)$, where $\varrho$ is given in \eqref{varrho}.

\begin{proposition} \label{bmwab} Let $V$ be  the natural representation of the quantum group  $\mathbf U_\kappa(\mathfrak g) $ associated   with $\mathfrak g\in \{\mathfrak {so}_{2n+1}, \mathfrak {sp}_{2n}, \mathfrak {so}_{2n}\}$.
For any $d\in\mathscr D_f$ and $\t\in\Std(\lambda')$ with $(f, \lambda)\in \Lambda_r$,      define $$v_{\lambda,\t, d}=v_\lambda E^f T_{w_\lambda}\n_{\lambda'}T_{d(\t)}T_{d}\in V^{\otimes r}.$$
If $\mathfrak g=\mathfrak{sp}_{2n}$ with $n\ge r$ or $\mathfrak g\in \{\mathfrak{so}_{2n}, \mathfrak {so}_{2n+1}\}$ with $n>r$,
 then  \begin{enumerate} \item the set $\{v_{\lambda,\t,d}\mid \t\in\Std(\lambda'), d\in\mathscr D_f\} $ is a basis of $\kappa$-space consisting of all highest weight vectors of $V^{\otimes r}$ with  weight $\sum_{i=1}^n \lambda_i\epsilon_i$;\item If $v\in V^{\otimes r}$ is a  highest weight vector with weight  $\lambda=\sum_{i=1}^n \lambda_i\epsilon_i$, then $\lambda$ is a partition of
 $r-2f$ for some non-negative integer $f$ such that $(f, \lambda)\in \Lambda_r$.
   \end{enumerate}
\end{proposition}

\begin{proof}
Obviously, both  $v_{\lambda,\t, d}$ and $v_\lambda$ have the same weight  $\sum _{i=1}^n \lambda_i\epsilon_i$.
 By Corollary~\ref{1dim} and \eqref{ei}, $$ \sum_{j=0}^{2f-1}k_i^{\otimes j}\otimes e_i\otimes 1^{\otimes r-j-1} \left(v_1\otimes v_{{1'}}\otimes \cdots \otimes v_1\otimes v_{{1'}}\otimes  v_{\mathbf i_\lambda} E^f T_{w_\lambda} \n_{\lambda'}\right)=0.$$
  Suppose $1\le k\le n$. By Lemma~\ref{basicl},  $e_i$   acts on $v_k$ via the corresponding formulae
for $\mathbf U_\kappa(\mathfrak{sl}_n)$ if $i\neq n$.  Moreover $e_n v_k=0$. By \eqref{ei},  $ v_\lambda   h=0$ for any  $h\in \mathscr B_{r}(\varrho, q)^{f+1}$. Via \cite[Corollary~3.4]{Enyang}, one can consider $T_{w_\lambda} \n_{\lambda'}$ in $v_\lambda  E^f T_{w_\lambda} \n_{\lambda'}$ as the corresponding element in Hecke algebra $\mathscr H_{r-2f}$ generated by $g_i, 2f+1\le i\le r-1$. By
Lemma~\ref{hightest of a},  we have   \begin{equation}\label{key12} \sum_{j=2f}^{r-1}k_i^{\otimes j}\otimes e_i\otimes 1^{\otimes r-j-1} \left(v_1\otimes v_{{1'}}\otimes \cdots \otimes v_1\otimes v_{{1'}}\otimes  v_{\mathbf i_\lambda} E^f T_{w_\lambda} \n_{\lambda'}\right)=0.\end{equation} So, $v_\lambda  E^f T_{w_\lambda} \n_{\lambda'}$  is killed by $e_i$, $1\le i\le n$.
In order to prove (1), it is enough to prove
$\text{ann}(v_\lambda)\bigcap M=0$, where
$$ M=\kappa\text{--span }\{ E^fT_{w_\lambda}\n_{\lambda'}T_{d(\t)}T_{d}\mid \t\in\Std(\lambda'),d\in\mathscr D_f\}.$$
If $x\in\text{ann}(v_\lambda)\cap  M$, then $$x=E^f\sum_{d\in\mathscr D_f}\sum_{\t\in\Std(\lambda')}a_\t T_{w_\lambda}\n_{\lambda'}T_{d(\t)} T_d$$ for some $a_\t\in \kappa$ and $v_\lambda x=0$.
By arguments  similar to those for Steps 1--2 in \cite[Lemma~5.18]{Hu},
$\{ v_\lambda E^f z_d T_d \mid d\in\mathscr D_f, z_d\neq 0\}$
is linearly independent, where $z_d=\sum_{\t\in\Std(\lambda')}a_\t T_{w_\lambda}\n_{\lambda'}T_{d(\t)}$.
\footnote{Although Hu proves the result for $\mathbf U_\kappa(\mathfrak{sp}_{2n})$, his arguments can be used smoothly for both  $\mathbf U_\kappa(\mathfrak {so}_{2n})$ and  $\mathbf U_\kappa(\mathfrak {so}_{2n+1})$. The key point is that Hu's arguments depend on \cite[(5.13)]{Hu} and does not depend on the explicit formulae for $(v_k\otimes v_l) \breve{R}$. So, we can use Corollary~\ref{fact} instead of  \cite[(5.13)]{Hu}.}
 In particular, we have $v_{\mathbf i_\lambda}z_d=0$ for any fixed $d$. By Lemma~\ref{hightest of a},   $a_\t=0$ for all $\t\in \Std(\lambda')$ and hence $x=0$.  So, $\{v_{\lambda,\t,d}\mid \t\in\Std(\lambda'), d\in\mathscr D_f\} $ is $\kappa$-linear independent.

 We identify $\lambda$ with $\sum_{i=1}^n \lambda_i \epsilon_i$.
Let  $\Delta(\lambda)$ be the Weyl module of  $\mathbf U_\kappa(\mathfrak g)$ with highest weight $\lambda$.  Since $V^{\otimes r}$ is a tilting module, by \cite[Lemma~5.1]{AGZ},  the dimension of $\Hom_{\mathbf U_\kappa(\mathfrak g)} (\Delta(\lambda), V^{\otimes r})$ is independent of $\kappa$. So, we can assume $\kappa=\mathbb C(v)$ and $v$ is an indeterminate  when we calculate  the dimension of $\Hom_{\mathbf U_\kappa(\mathfrak g)} (\Delta(\lambda), V^{\otimes r})$. In this case, $V^{\otimes r}$ is completely reducible.  Since we are assume   $\mathfrak g=\mathfrak{sp}_{2n}$ with $n\ge r$ or $\mathfrak g\in \{\mathfrak{so}_{2n}, \mathfrak {so}_{2n+1}\}$ with $n>r$,  the multiplicity of irreducible $\mathbf U_\kappa(\mathfrak g)$-module $L_\lambda$ (which is $\Delta(\lambda)$ in this case)  is equal to the number of so-called up-down tableaux of type $\lambda$ (see, e.g. \cite[(5.5)]{LR}).
 Such a number is equal to the dimension of $C(f, \mu)$ with $\mu\in \{\lambda, \lambda'\}$ (see \cite{Enyang}). Thus, the cardinality of $\{v_{\lambda,\t,d}\mid \t\in\Std(\lambda'), d\in\mathscr D_f\} $ is equal to
the dimension of $\Hom_{\mathbf U_\kappa} (\Delta(\lambda), V^{\otimes r})$.

 Suppose $v\in V^{\otimes r}$ is a highest weight vector  with  weight $\lambda$.
By the universal property of Weyl modules,  there is  an epimorphism  from
$\Delta(\lambda)$ to $\mathbf U_k(\mathfrak g)v$. It gives rise to an $f_v\in \Hom_{\mathbf U_\kappa(\mathfrak g)} (\Delta(\lambda), V^{\otimes r})$ sending the highest weight vector $\mathbf v$ of $\Delta(\lambda)$ to $v$.
In particular, we have $f_{\lambda, \t, d}$ sending $\mathbf v$ to $v_{\lambda,\t,d}$ such that  $\{f_{\lambda,\t,d}\mid \t\in\Std(\lambda'), d\in\mathscr D_f\} $
is a basis of  $\Hom_{\mathbf U_\kappa} (\Delta(\lambda), V^{\otimes r})$. This implies (1).

If there is a highest weight vector $v\in V^{\otimes r}$ with  weight $\mu$, then there is an epimorphism from $\Delta(\mu)$ to $\mathbf U_\kappa v$ and hence $\dim  \Hom_{\mathbf U_\kappa(\mathfrak g)} (\Delta(\mu), V^{\otimes r})\neq 0$. Since  $V^{\otimes r}$ is a tilting module, by \cite[Lemma~5.1]{AGZ}, such a dimension is independent of $\kappa$. So, we assume $\kappa=\mathbb C(v)$. In this case, $V^{\otimes r}$ is completely reducible. By \cite[(5.5)]{LR}, $\mu=\sum_{i=1}^n \mu_i\epsilon_i$ such that $(f, \mu)\in \Lambda_r$ and (2) follows.   \end{proof}

 Abusing of notation, we denote $\sum_{i=1}^n \lambda_i \epsilon_i$ by $\lambda$. In the remaining part of this section, we denote by  $\nabla(\lambda)$  the  co-Weyl module of $\mathbf U_\kappa(\mathfrak g)$ with respect to the   highest weight $\lambda$.

We always keep assumptions that either  $\mathfrak g=\mathfrak{sp}_{2n}$ with $n\ge r$ or $\mathfrak g\in \{\mathfrak{so}_{2n}, \mathfrak {so}_{2n+1}\}$ with $n>r$.
Let $V$ be the natural representation of $\mathbf U_\kappa(\mathfrak g)$. For any $\mathbf U_{\kappa}(\mathfrak{g})$-module $M$, $\Hom_{\mathbf U_{\kappa}(\mathfrak{g})}(M, V^{\otimes r})$   is a right $\mathscr B_{r}(\varrho, q)$-module in the sense \begin{equation} \label{lrm} (\varphi b) (x)= \varphi(x)\alpha (b), \end{equation}
  for all $ x \in M$,  $b\in\mathscr B_r(\varrho, q)$, and $ \varphi\in \Hom_{\mathbf U_{\kappa}(\mathfrak{g})}(M, V^{\otimes r})$ where $\alpha$ is  the automorphism $\gamma$  (resp.,  identity automorphism)  given in Lemma~\ref{anti}(2) if $\mathfrak g=\mathfrak{so}_{2n+1}$ (resp.,  $\mathfrak g\in \{\mathfrak{so}_{2n}, \mathfrak {sp}_{2n}\}$).
   We remark that  $\Hom_{\mathbf U_{\kappa}(\mathfrak{g})} (M, V^{\otimes r})$ can be considered as a left $\mathscr B_r(\varrho, q)$-module such that $x f:=f \sigma(x)$, $\forall x\in \mathscr B_r(\varrho, q)$
and $f\in  \Hom_{\mathbf U_{\kappa}(\mathfrak{g})} (M, V^{\otimes r})$, where $\sigma$ is the anti-involution on $\mathscr B_r(\varrho, q)$ defined in Lemma~\ref{anti}.
  Let $\mathbf U_\kappa(\mathfrak g)$-mod (resp., mod-$\mathscr B_r(\varrho, q)$) be the category of finite dimensional left $\mathbf U_\kappa(\mathfrak g)$-modules (resp., right $\mathscr B_r(\varrho, q)$-modules) over $\kappa$.
Later on, we define  \begin{equation}\label{Fun}  \mathcal F=\Hom_{\mathbf U_{\kappa}(\mathfrak{g})}(-, V^{\otimes r}).\end{equation}

\begin{proposition}\label{cell-weyl} Suppose $(f, \lambda)\in \Lambda_r$. 
\begin{enumerate}
\item If  $\mathfrak g=\mathfrak{sp}_{2n}$ with $n\ge r$, then $\mathcal F(\Delta(\lambda))\cong C(f,\lambda')$ as right $\mathscr B_r(\varrho, q)$-modules.
\item If  $\mathfrak g=\mathfrak{so}_{2n}$ with $n>r$, then $\mathcal F(\Delta(\lambda))\cong C(f,\lambda')$ as right $\mathscr B_r(\varrho, q)$-modules.
\item If  $\mathfrak g=\mathfrak{so}_{2n+1}$ with $n> r$, then $\mathcal F(\Delta(\lambda)) \cong \tilde{C}(f,\lambda')$ as right $\mathscr B_r(\varrho, q)$-modules.
\end{enumerate}
\end{proposition}

\begin{proof} It is routine to prove that $$C(f, \lambda')\cong E^f \m_\lambda T_{w_\lambda} \n_{\lambda'} \mathscr B_r(\varrho, q) \pmod{  \mathscr B_r(\varrho, q)^{f+1}},$$ where $ \mathscr B_r(\varrho, q)^{f+1}$ is the two-sided ideal of  $\mathscr B_r(\varrho, q) $ generated by $E^{f+1}$. Further, as a $\kappa$-space,  $ E^f \m_\lambda T_{w_\lambda} \n_{\lambda'} \mathscr B_r(\varrho, q) \pmod{  \mathscr B_r(\varrho, q)^{f+1}}$ has a $\kappa$-basis
consisting of $E^f\m_\lambda T_{w_\lambda}\n_{\lambda'}T_{d(\t)}T_d\pmod{\mathscr B_r(\varrho, q)^{f+1}}$,   for all $(\t,d)\in \Std(\lambda')\times\mathscr D_f$. Let
$f_{\lambda,\t,d}\in \mathcal F(\Delta(\lambda))$ sending the highest weight vector of $\Delta(\lambda)$ to $v_{\lambda, \t, d}\in V^{\otimes r}$ in Proposition~\ref{bmwab}. Then   $f_{\lambda,\t,d}$'s are $\kappa$-base elements of $\mathcal F(\Delta(\lambda))$. It is routine to check that the required isomorphism $\Phi$ in (1)  is the  $\kappa$-linear isomorphism  satisfying   $$\Phi(f_{\lambda,\t,d})=E^f\m_\lambda T_{w_\lambda}\n_{\lambda'}T_{d(\t)}T_d+\mathscr B_r(\varrho, q)^{f+1}. $$
Finally, (2)--(3) can be proved similarly. The reason why we use right cell module   $\tilde{C}(f,\lambda')$ in (3) is that we use usual linear  dual in Corollary~\ref{duforb} when $\mathfrak g=\mathfrak {so}_{2n+1}$.
\end{proof}

  For each left $\mathbf U_\kappa(\mathfrak g)$-module $M$, $\Hom_{\mathbf U_\kappa (\mathfrak g)}(V^{\otimes r},M)$
is a left $\mathscr B_r(\varrho, q)$-module such that,   for any $ x\in V^{\otimes r}$, $b\in\mathscr B_r(\varrho, q)$ and $\phi\in\Hom_{\mathbf U_\kappa (\mathfrak g)}(V^{\otimes r},M)$,
\begin{equation}\label{lbrt} (b\phi)(x)=\phi(xb).\end{equation}
Also,  $V^{\otimes r}\otimes_{\mathscr  B_r(\varrho, q)}N$ is a left  $\mathbf U_\kappa (\mathfrak g)$-module for any left $\mathscr B_r(\varrho, q)$-module $N$.
In the following, let  $\mathscr B_r(\varrho, q)\text{-mod}$ be the category of left $\mathscr B_r(\varrho, q)$-modules.

\begin{definition}\label{mfg} Let  $\mathbf f$ and $\mathbf g$ be two functors
$$\begin{aligned}\mathbf f: \mathbf U_\kappa (\mathfrak g)\text{-mod} &\longrightarrow \mathscr B_r(\varrho, q)\text{-mod}\\
M &\longmapsto \Hom_{\mathbf U_\kappa(\mathfrak g)}(V^{\otimes r},M), \\
\mathbf g:\mathscr B_r(\varrho, q) \text{-mod} &\longrightarrow \mathbf U_\kappa (\mathfrak g)\text{-mod}\\
N &\longmapsto V^{\otimes r}\otimes_{\mathscr B_r(\varrho, q)}N.
\end{aligned} $$\end{definition}

It follows from \cite[Theorem~2.11]{JR} that  $\mathbf f$ and $\mathbf g$ are adjoint pairs in the sense that
\begin{equation}\label{ajoint}
\Hom_{\mathbf U_\kappa(\mathfrak g)}(\mathbf g(N),M)\cong \Hom_{\mathscr B_r(\varrho, q)}(N,\mathbf f(M)),
\end{equation}
   as $\kappa$-spaces  where $M$ (resp., $N$) is a left $\mathbf U_\kappa (\mathfrak g)$-module (resp.,  left $\mathscr B_r(\varrho, q)$-module $N$).

\begin{lemma}\label{g} Let  $T$ be  an indecomposable direct summand of the left $\mathbf U_\kappa (\mathfrak g)$-module $V^{\otimes r}$.  Then $\mathbf g\mathbf f(T)\cong T $.
\end{lemma}
\begin{proof}By Theorem~\ref{isom},  $ \mathbf f(V^{\otimes r})\cong \mathscr B_r(\varrho, q)$  and hence   $\mathbf  g\mathbf f(V^{\otimes r})\cong V^{\otimes r}$.
The corresponding  isomorphism $\phi$ sends $v\otimes b$ to $b(v)$ for any $v\in V^{\otimes r}$ and $b\in \mathbf f(V^{\otimes r})$. Since $T$ is a direct summand of $V^{\otimes r}$,
the projection  $\pi:  V^{\otimes r}\rightarrow T$ induces a homomorphism $1\otimes \mathbf f(\pi)$ from $\mathbf g\mathbf f(V^{\otimes r})$ to $\mathbf g\mathbf f({T})$
such that $\pi\circ \phi=\tilde \phi\circ ( 1\otimes \mathbf f(\pi))$ where $\tilde \phi$ is the homomorphism from $\mathbf g \mathbf  f(T)$ to $T$ sending $v\otimes h$ to $h(v)$ where $v\in V^{\otimes r}$ and $h\in \mathbf f(T)$. So, $\tilde \phi$ is  surjective. Comparing dimensions yields that $\tilde \phi$ is an isomorphism.
\end{proof}

We remark that any right $\mathscr B_r(\varrho, q)$-module  can be considered as left $\mathscr B_r(\varrho, q)$-module  via the  anti-involution $\sigma$ in Lemma~\ref{anti} and vice versa.
In Theorem~\ref{schur}, let $\omega_0$ be   the longest element of the Weyl group associated to $\mathfrak g$. 
In the remaining part of this paper, let $T(\lambda)$ be the indecomposable (or partial)  tilting module of $\mathbf U_\kappa(\mathfrak g)$  with respect to the highest weight $\lambda$. 
Let $(T(\lambda):\Delta(\mu)) $ be the multiplicity of  the Weyl module $\Delta(\mu)$ in  $T(\lambda)$. This multiplicity is well-defined since it is independent of    Weyl filtrations of $T(\lambda)$.

\begin{theorem}\label{schur} Let $V$ be the natural representation of $\mathbf U_\kappa(\mathfrak g)$ such that
 $\mathfrak g=\mathfrak{sp}_{2n}$   (resp.,  $\mathfrak g\in \{\mathfrak{so}_{2n}, \mathfrak {so}_{2n+1}\}$)  with $n\ge r$ (resp., $n>r$).
  Let $\mathscr B_r(\varrho, q)$ be the Birman-Murakami-Wenzl algebra over $\kappa$, where  $\varrho$ is given in \eqref{varrho}.
\begin{enumerate}\item Any partial tilting module which appears as an indecomposable direct summand of $V^{\otimes r}$ is of
form $T(\lambda)$ for some $(f, \lambda')\in \Lambda_{r}$ with  $\lambda$ being $e$-regular.
\item  Suppose $\mathfrak g=\mathfrak{so}_{2n}, \mathfrak{sp}_{2n}$.  As right $\mathscr B_r(\varrho, q)$-modules,
\begin{itemize} \item [(a)]  $\mathbf f(\nabla(\mu))\cong \mathcal F(\Delta(\mu))\cong C(f, \mu')$ for any  $(\ell, \mu)\in \Lambda_{r}$,
\item [(b)] $\mathbf f(T(\lambda))\cong P(f,\lambda')$  for any $(f, \lambda')\in \Lambda_{r}$ with  $\lambda$ being $e$-regular.
\item [(c)]  $(T( \lambda): \Delta({ \mu}))=[C(\ell, \mu'): D^{f, \lambda'}]$.  \end{itemize}
 \item Suppose   $\mathfrak g=\mathfrak{so}_{2n+1}$.   As right $\mathscr B_r(\varrho, q)$-modules, \begin{itemize} \item [(a)]  $\mathbf f(\nabla(-\omega_0\mu))\cong \mathcal F(\Delta(\mu))\cong \tilde C(f, \mu')$, for $(\ell, \mu)\in \Lambda_{r}$,
\item [(b)] $\mathbf f(T(-\omega_0\lambda))\cong \tilde{P}(f,\lambda')$ for any $(f, \lambda')\in \Lambda_{r}$ with  $\lambda$ being $e$-regular.
\item [(c)] $(T(-\omega_0 \lambda): \Delta({ -\omega_0\mu}))=[\tilde C(\ell, \mu'): \tilde D^{f, \lambda'}]$. \end{itemize}\end{enumerate}
\end{theorem}
\begin{proof}
Suppose $\mathfrak g\in \{\mathfrak{so}_{2n}, \mathfrak{sp}_{2n}\}$. Let  $ \Psi:  \mathcal F(\Delta(\mu)) \rightarrow  \Hom_{\mathbf U_\kappa (\mathfrak g)}((V^{\otimes r})^\circ,\Delta(\mu)^\circ)
$  be the $\kappa$-linear isomorphism  such that
\begin{equation} \label{billl} [\Psi(\phi)(v^\circ )](x)=\langle v,\phi(x)\rangle=v^\circ (\phi(x)), \forall  \phi\in \mathcal F(\Delta(\mu)), v\in V^{\otimes r}, x\in \Delta(\mu) \end{equation} where  $\langle \ ,\ \rangle $ is the bilinear form  defined in \eqref{mf1}  and $v^\circ \in \Hom_\kappa (V^{\otimes r}, \kappa )$ is defined in \eqref{duall}. So, for any $b\in \mathscr B_r(\varrho, q)$,
\begin{equation} \label{eq1} \begin{aligned}   (\Psi(b\phi)(v^\circ ))  (x) & = \langle v,b\phi(x)\rangle,   \ \text{ by \eqref{billl}}, \\
& =\langle v,\phi(x)\sigma(b)\rangle   \text{ by \eqref{lrm}},\\ & = (v^\circ  b)(\phi(x))   \text{ by  \eqref{contramod1} and \eqref{duall},}    \\ &=[\Psi(\phi)(v^\circ b)]( x), \text{ by \eqref{billl}.}  \\
\end{aligned} \end{equation}
and hence   by \eqref{lbrt}, $\Psi(b\phi)(v^\circ )= \Psi(\phi)(v^\circ b)=(b(\Psi(\phi)))(v^\circ)$.
So $\Psi(b\phi)=b\Psi(\phi)$.  By  Corollary~\ref{cond}, and  $\nabla(\mu)\cong \Delta(\mu)^\circ$ for any dominant integral weight $\mu$ (see \cite[Proposition~4.1.6]{Dk}), $\mathbf f(\nabla(\mu))\cong \mathcal F(\Delta(\mu))$ as left $\mathscr B_r(\varrho, q)$-modules. Via anti-involution, it can be considered as isomorphism for right $\mathscr B_r(\varrho, q)$-modules.  Finally, the last isomorphism in 2(a) follows from Proposition~\ref{cell-weyl}.

By \cite[II,Proposition 2.1(c)]{ARS}, the functor $\mathbf f$ in Definition~\ref{mfg} induces a category equivalence between  direct sums of direct summands of
the left $ \mathbf U_\kappa (\mathfrak g)$-module $V^{\otimes r}$ and direct sums of direct summands of left  $\mathscr B_r(\varrho, q)$-module  $\mathscr B_r(\varrho, q)$. So $ \mathbf f(T(\mu))$ is a principal  indecomposable  left $\mathscr B_r(\varrho, q)$-module if
$T(\mu)$ is an indecomposable direct summand of $V^{\otimes r}$. By the universal property, the Weyl module $\Delta(\mu)$ is  a submodule of  $T(\mu)$. So, $\dim  \Hom_{\mathbf U_\kappa(\mathfrak g)} (\Delta(\mu), V^{\otimes r})\neq 0$. By arguments in the proof of Proposition~\ref{bmwab}(2), $\mu$ has to be a partition such that $(\ell, \mu)\in \Lambda_r$ for some $\ell$.   So,
 $(\ell, \mu')\in \Lambda_r$.
For any $(k,\nu')\in\Lambda_{r}$, we have   $\kappa$-linear isomorphism
\begin{equation} \label{decom-hom}\begin{aligned}
 \Hom_{\mathbf U_\kappa (\mathfrak g)} (T(\mu),\nabla(\nu)) & \cong \Hom_{\mathbf U_\kappa (\mathfrak g)} (\mathbf g\mathbf f(T(\mu)),\nabla(\nu)), \text{ by  Lemma~\ref{g}},\\
& \cong  \Hom_{\mathscr B_r(\varrho, q)} (\mathbf f (T(\mu)) , \mathbf f(\nabla(\nu))),   \text{ by \eqref{ajoint} },\\
&\cong
\Hom_{\mathscr B_r(\varrho, q)}( P(f,\lambda'),C(k,\nu')), \text{ by 2(a), }\\
\\ \end{aligned}
\end{equation}
for some $(f, \lambda')\in \Lambda_{r}$ such that $\mathbf f(T( \mu))=P(f,\lambda')$ with $\lambda$ being $e$-regular.
Note that $ P(f,\lambda')$ is a principal indecomposable  module, we have
\begin{equation}\label{decom-hom1}   \dim_\kappa\Hom_{\mathscr B_r(\varrho, q)}( P(f,\lambda'),C(k,\nu'))=[C(k,\nu'): D^{f, \lambda'}].\end{equation}
If we assume $\nu=\mu$, then  $(\ell,\mu')\unrhd (f,\lambda')$.   If we assume  $\nu=\lambda$, then  $\Hom_{\mathbf U_\kappa (\mathfrak g)}(T(\mu),\nabla(\lambda))\neq 0$.
Since $\mu$ is the highest weight of  $T(\mu)$,
  $ \lambda\unlhd  \mu$ and   either $f>\ell$ or $f=\ell$,  forcing  $(\ell,\mu')\unlhd (f,\lambda')$. So,   $f=\ell$ and $\mu=\lambda$.
This proves 2(b) as left $\mathscr B_r(\varrho, q)$-modules. Via anti-involution $\sigma$ in Lemma~\ref{anti}, we have 2(b) as right  $\mathscr B_r(\varrho, q)$-modules.  In particular, we have proved  (1) for $\mathfrak g\in \{\mathfrak{so}_{2n}, \mathfrak {sp}_{2n}\}$. Finally, 2(c) follows from \eqref{decom-hom}-\eqref{decom-hom1}
and $(T( \lambda): \Delta( \mu))=\dim_\kappa \Hom_{\mathbf U_\kappa (\mathfrak g)}(T(\mu),\nabla(\nu))$.

 Suppose  $\mathfrak g=\mathfrak{so}_{2n+1}$. Let  $ \Psi:  \mathcal F(\Delta(\mu))\rightarrow  \Hom_{\mathbf U_\kappa(\mathfrak g)}((V^{\otimes r})^*,\Delta(\mu)^*)$ be the $\kappa$-linear isomorphism
  such that
\begin{equation} \label{billll} \Psi(\phi)(v^*)(x)= \langle v,\phi(x)\rangle , \forall  \phi\in \mathcal F(\Delta(\mu)), v\in V^{\otimes r}, x\in \Delta(\mu), \end{equation}
 where  $\langle \ , \ \rangle $ is defined \eqref{bi12} and $v^*$ is defined in a natural way. By arguments similar to those for 2(a)-(b), one can check that $\Psi$ is a left $\mathscr B_r(\varrho, q)$-isomorphism.
 So, 3(a) follows from   Corollary~\ref{duforb}, Proposition~\ref{cell-weyl} and the fact $\nabla(-\omega_0\mu)\cong \Delta(\mu)^*$ (see \cite[Proposition~3.3]{Apw}).
Finally, 3(b)-(c) and (1) for  $\mathfrak{so}_{2n+1}$ follow from arguments similar to those for 2(b) and (1) for $\mathfrak {so}_{2n}$ and $\mathfrak {sp}_{2n}$. \end{proof}

We close the paper by giving the following remarks on decomposition numbers of $\mathscr B_r(\varrho, q)$ over $\mathbb C$.

 \begin{remark}\begin{enumerate} \item Suppose that  $\varrho\not\in \{q^a, -q^a\mid a\in \mathbb Z\}$. In \cite{RSsin}, Rui and  Si have proved that $\mathscr B_r(\varrho, q)$ is Morita equivalent to $\bigoplus_{i=0}^{\lfloor r/2\rfloor} \mathscr H_{r-2i}$ over $\kappa$. \begin{itemize} \item[(a)] If  $q^2$ is not a root of unity,
 $\mathscr B_r(\varrho, q)$  is split semisimple \footnote{ A necessary and sufficient condition on the semisimplicity of $\mathscr B_r(\varrho, q)$ has been given in  \cite{RSBMW}. See also \cite{Wen} for some partial results over $\mathbb C$.} and the decomposition matrix of $\mathscr B_r(\varrho, q)$  is the identity matrix.
 \item[(b)] If $q^2$ is a root of unity and $\kappa$ is $\mathbb C$, by Ariki's famous results on LLT conjecture  in \cite{Ari},
  decomposition numbers of  $\mathscr B_r(\varrho, q)$ are determined by values of certain  inverse Kazhdan-Lusztig polynomials associated with some extended affine Weyl groups of type $A$ at $q=1$. In this case,  there is no restriction on $e$, the order of $q^2$.\end{itemize}
\item Suppose $\varrho\in \{-q^{ a}, q^a\mid a\in \mathbb Z\}$. \begin{itemize} \item[(a)] If $q^2$ is not a root of unity,
    Xu showed  that $\mathscr B_r(\varrho, q)$ is multiplicity free over $\mathbb C$~\cite{Xu}.
\item[(b)] If $\kappa=\mathbb C$ and $o(q^2)=e$, we assume  $q=\exp(2\pi i/e)$ if $e$ is odd and
     $q=\exp(\pi i/e)$ if $e$ is even. Further, we assume that  $q^{1/2}=\exp(\pi i/2e)$ if $e$ is even.  In this case, $q^{1/2}$ is a primitive $4e$-th root of unity. If $e$ is odd, $-q^{2n}\in \{-q^{2k+1}\mid k\in \mathbb Z\}$. If $e$ is even, $q^e=-1$ and $-q^{2n}=q^{2n+e}$. Finally, if $\varrho=q^{2n}$ and $e$ is odd, $\varrho=q^{2n+e}$. In summary, when we calculate decomposition numbers of $\mathscr B_r(\varrho, q)$
     for  $\varrho\in \{-q^{ a}, q^a\mid a\in \mathbb Z\}$ and $q^2$ being  a root of unity,   we can always assume that $\varrho$'s are given in \eqref{varrho}. Moreover, we can assume $e$ is even if $\varrho=q^{2n}$ for some $n\in \mathbb Z$.
    By  Theorem~\ref{schur},  decomposition numbers of $\mathscr B_r(\varrho, q)$ are determined by  multiplicities of Weyl modules in certain indecomposable tilting modules for $\mathbf U_\kappa(\mathfrak g)$.  Soergel~\cite{Soe}
     has described multiplicities of Weyl modules in certain indecomposable tilting modules for $\mathbf U_\kappa(\mathfrak g)$ via
      the equivalence of categories between modules for quantum groups at roots of unity and corresponding module categories for Kac-Moody algebras in \cite{Soe1}. Due to \cite{KL1}, this equivalence  is only proved when $e \geq 29$.  In principal, we know decomposition numbers  of $\mathscr B_r(\varrho, q)$ for $\varrho\in \{-q^a, q^a\mid a\in \mathbb Z\}$
      over $\mathbb C$ only if  $e\ge 29$.\end{itemize}\end{enumerate}\end{remark}

\providecommand{\bysame}{\leavevmode ---\ } \providecommand{\og}{``}
\providecommand{\fg}{''} \providecommand{\smfandname}{and}
\providecommand{\smfedsname}{\'eds.}
\providecommand{\smfedname}{\'ed.}
\providecommand{\smfmastersthesisname}{M\'emoire}
\providecommand{\smfphdthesisname}{Th\`ese}

\end{document}